\documentclass[a4paper,11pt,reqno]{amsart}

\usepackage{amssymb}
\usepackage[all]{xy}

\usepackage{hyperref}
\usepackage{xcolor}
\usepackage{enumitem}
\usepackage{subcaption}
\usepackage[final]{graphicx}
\usepackage{float}
\usepackage{epic}
\usepackage{setspace}

\usepackage[all]{xy}
\usepackage{color}

\usepackage{verbatim}

\usepackage{tikz}
\usetikzlibrary[topaths]

\newcount\mycount

\numberwithin{equation}{section}

\numberwithin{equation}{section}
\newtheorem{theorem}{Theorem}[section]

\newtheorem{lemma}[theorem]{Lemma}

\newtheorem{proposition}[theorem]{Proposition}

\newtheorem{corollary}[theorem]{Corollary}
\newtheorem*{theorem*}{Theorem}
\newtheorem*{corollary*}{Corollary}

\theoremstyle{definition}
\newtheorem{definition}[theorem]{Definition}

\theoremstyle{remark}
\newtheorem{example}[theorem]{Example}

\theoremstyle{remark}

\newtheorem{remark}[theorem]{Remark}

\newcommand\bp{\begin{proof}}
\newcommand\ep{\end{proof}}

\begin{document}

 \title{Simplicity of right-angled Hecke C$^{\ast}$-algebras}

\author{Mario Klisse}

\address{TU Delft, EWI/DIAM,
	P.O.Box 5031,
	2600 GA Delft,
	The Netherlands}

\email{m.klisse@tudelft.nl}

\begin{abstract}

By exploiting properties of boundaries associated with Coxeter groups we obtain a complete characterization of simple right-angled multi-parameter Hecke C$^{\ast}$-algebras. This extends previous results by Caspers, Larsen and the author. Based on work by Raum and Skalski, we further describe the center and the character space of right-angled Hecke C$^{\ast}$-algebras.

\end{abstract}

\date{\today. \emph{MSC2010:} 20C08, 20F55, 46L05, 46L65. The author is supported by the NWO project ``The structure of Hecke-von Neumann algebras'', 613.009.125.}

\maketitle


\section*{Introduction}

(Iwahori) Hecke algebras are deformations of the group algebra of Coxeter groups depending on a deformation (multi-)parameter $q$. They can be viewed as an abstraction of certain endomorphism rings which naturally appear in the representation theory of finite groups of Lie type \cite{Bourbaki} and are particularly well studied in the case of spherical and affine Coxeter groups (see \cite{IwahoriMatsumoto}, \cite{KL2}, \cite{Bernstein}, \cite{KL}). For other Coxeter groups they appear in the context of buildings and Kac-Moody groups acting on them \cite{Remy}.

Hecke algebras of a given Coxeter system $(W,S)$ can be naturally represented on the Hilbert space $\ell^{2}(W)$ of square-summable functions on $W$. They complete to C$^{\ast}$-algebras $C_{r,q}^{\ast}(W)$ and von Neumann algebras $\mathcal{N}_{q}(W)$. The study of these operator algebras gave insight in the cohomology of associated buildings and its $\ell^2$-Betti numbers (see \cite{Dymara}, \cite{Dymara2}) and they are related to Dykema's interpolated free group factors, which play an important role in the treatment of the infamous free factor problem (see \cite{Dykema}, \cite{Radulescu}, \cite{Gar}). Much earlier Hecke operator algebras of spherical and affine type Coxeter systems have been studied in \cite{Matsumoto}.

Despite their ubiquitousness, until now Hecke operator algebras are well understood only in the case of spherical and affine type Coxeter groups. In particular, the natural question for a characterization of the simplicity (i.e. triviality of the ideal structure) and the trace-uniqueness of $C_{r,q}^{\ast}(W)$ and the factoriality (i.e. triviality of the center) of $\mathcal{N}_{q}(W)$ is wide open. Recently, the investigation of right-angled Hecke C$^{\ast}$-algebras and right-angled Hecke-von Neumann algebras made some progress. In \cite{Gar} Garncarek characterized the factoriality of single-parameter Hecke-von Neumann algebras. Complementing his ideas with a new combinatorial approach, the result was later extended to the multi-parameter case by Raum and Skalski \cite{Raum}. In \cite{Mario} Caspers, Larsen and the author studied the C$^{\ast}$-algebraic setting and proved, using classical averaging techniques, simplicity and trace-uniqueness results for right-angled Hecke C$^{\ast}$-algebras and certain ranges of deformation parameters $q$. As remarked in \cite[Subsection 5.4]{Mario}, Dykema's results on free products of finite dimensional abelian C$^{\ast}$-algebras in \cite{Dykema2} imply a complete description of the ideal structure and the trace-uniqueness of Hecke C$^{\ast}$-algebras of free products of right-angled abelian Coxeter groups. It is further known that the reduced group C$^{\ast}$-algebra of an irreducible Coxeter system is simple if and only if the corresponding Coxeter system is of non-affine type (see \cite{Fe}, \cite{DeLaHarpe}, \cite{Cornulier}). Other relevant references treating non-affine Hecke operator algebras are \cite{Caspers}, \cite{CSW}, \cite{Raum2}.

A notion that was famously used by Kalantar and Kennedy \cite{KalantarKennedy} to solve the longstanding question regarding the C$^{\ast}$-simplicity of a given (discrete) group is that of Furstenberg's boundary and (topological) boundary actions. The authors established a link between dynamical properties of the Furstenberg boundary of a given group and the structure of the corresponding group C$^{\ast}$-algebra, which led to  results on simplicity, uniqueness of trace and tightness of nuclear embeddings of group C$^{\ast}$-algebras (see e.g. \cite{Haagerup}, \cite{BKKO}, \cite{Kennedy}) and inspired various generalizations (see e.g. \cite{BK}, \cite{HaKa}, \cite{Adam}). Inspired by the approach in \cite{Haagerup} our present work goes into a similar direction. In \cite{Mario2} the author introduced and studied topological boundaries and compactifications associated with connected rooted graphs. These are topological spaces that reflect combinatorial properties of the underlying graph and which are particularly tractable in the case of (Cayley graphs of) Coxeter groups. In the latter context the spaces have been considered earlier by Caprace-Lécureux \cite{Caprace} and Lam-Thomas \cite{Lam}. The striking advantage of the construction is its close connection to the Hecke operator algebras of the corresponding Coxeter system (see Section \ref{1} for more details), which has been utilized in \cite[Section 4]{Mario2}. We will further exploit the implications of this connection by using it to characterize the simplicity of right-angled Hecke C$^{\ast}$-algebras, thus extending the results in \cite{Mario} and (partially) answering \cite[Question 5.13]{Mario}. This leads to a full classification of the simplicity in the right-angled case which is the main result of this paper.

\begin{theorem*}
Let $(W,S)$ be an irreducible, right-angled Coxeter system with $\#S<\infty$ and $q=(q_{s})_{s\in S}\in\mathbb{R}_{>0}^{(W,S)}$ a multi-parameter. Let further $\mathcal{R}$ be the region of convergence of the growth series $\sum_{\mathbf{w}\in W}z_{\mathbf{w}}$, set $\mathcal{R}^{\prime}:=\{ (q_{s}^{\epsilon_{s}})_{s\in S}\mid q\in\mathcal{R}\cap\mathbb{R}_{>0}^{(W,S)}\text{, }\epsilon\in\left\{ -1,1\right\} ^{(W,S)}\}$  and define $\overline{\mathcal{R}^{\prime}}$ to be the closure of $\mathcal{R}^{\prime}$ in $\mathbb{R}_{>0}^{(W,S)}$. Then the Hecke C$^{\ast}$-algebra $C_{r,q}^{\ast}(W)$ is simple if and only if $q \in\mathbb{R}_{>0}^{(W,S)}\setminus\overline{\mathcal{R}^{\prime}}$.
\end{theorem*}

\begin{corollary*}
Let $(W,S)$ be an irreducible, right-angled Coxeter
system with $\#S=\infty$ and $q=(q_{s})_{s\in S}\in\mathbb{R}_{>0}^{(W,S)}$ a multi-parameter.
Then the Hecke C$^{\ast}$-algebra $C_{r,q}^{\ast}(W)$ is simple
if and only if there exists a finite subset $T\subseteq S$ such that
the Hecke C$^{\ast}$-algebra $C_{r,q_{T}}^{\ast}(W_{T})$ of the special subgroup $W_T \subseteq W$ with $q_{T}:=(q_{t})_{t\in T}$
is simple.
\end{corollary*}

Using a Haagerup-type inequality from \cite{Mario}, we will also prove that the central projections of right-angled Hecke-von Neumann algebras considered by Raum and Skalski in \cite{Raum} are already contained in the corresponding Hecke C$^{\ast}$-algebras. This leads to a decomposition of $C_{r,q}^{\ast}(W)$ which is analogous to the one of $\mathcal{N}_{q}(W)$ and can be used to characterize the space of characters (i.e. unital, multiplicative, linear functionals) of  $C_{r,q}^{\ast}(W)$.\\

\noindent \emph{Structure.} In Section \ref{1} we recall general facts about Coxeter groups, the boundaries introduced in \cite{Mario2} and multi-parameter Hecke algebras (resp. their operator algebraic counterparts). In Section \ref{2} a number of technical statements related to certain C$^{\ast}$-algebras associated with Coxeter systems are proved, which allow to translate group algebraic arguments into the $q$-deformed setting. This will be used in Section \ref{3} and \ref{simplicity} where the center, the character space and the simplicity of right-angled Hecke C$^{\ast}$-algebras are characterized.\\

\vspace{1mm}


\section{Preliminaries and notation} \label{1}

\subsection{General notation}

We will write $\mathbb{N}:=\left\{ 0,1,2,...\right\}$  and $\mathbb{N}_{\geq1}:=\left\{ 1,2,...\right\}$  for the natural numbers. Scalar products of Hilbert spaces are linear in the first variable and we denote the bounded operators on a Hilbert space $\mathcal{H}$ by $\mathcal{B}(\mathcal{H})$. For a C$^{\ast}$-algebra $A$ we will write $\mathcal{S}(A)$ for the state space of $A$ and endow it with the weak-$\ast$ topology. Further, if $(A,G,\alpha)$ is a dynamical system we write $g.a:=\alpha_{g}(a)$ where $a\in A$, $g\in G$. Similar notation is being used for (continuous) group actions on topological spaces. The symbol $\odot$ denotes the algebraic tensor product of $\ast$-algebras, $\otimes$ is the minimal tensor product of C$^{\ast}$-algebras, $\overline{\otimes}$ denotes the tensor product of von Neumann algebras and we write $\rtimes_{r}$ for reduced (C$^{\ast}$-algebraic) crossed products. Further, the neutral element of a group is always denoted by $e$ and for a set $S$ we write $\#S$ for the number of elements in $S$ and $\chi_S$ for the characteristic function on $S$.


\subsection{Coxeter groups}

A \emph{Coxeter group} $W$ is a group generated by a (possibly infinite) set $S$ of the form
\begin{eqnarray}
\nonumber
W=\left\langle S\mid\forall s,t\in S\text{: }(st)^{m_{st}}=e\right\rangle ,
\end{eqnarray}
where $m_{st}\in\left\{ 1,2,...,\infty\right\}$  with $m_{ss}=1$ and $m_{st}\geq2$ for all $s\neq t$. The condition $m_{st}=\infty$ means that no relation of the form $(st)^m = 1$, $m\in \mathbb{N}$ is imposed. The pair $(W,S)$ is called a \emph{Coxeter system}. It is \emph{right-angled} if $m_{st}=2$ or $m_{st}=\infty$  for all $s\neq t$. If the generating set $S$ is finite the system $(W,S)$ has \emph{finite rank}. The data of $(W,S)$ is usually encoded in its \emph{Coxeter diagram} whose vertex set is $S$ and whose edge set is given by $\{ (s,t)\mid m_{st}\geq3\}$  where every edge between two vertices $s,t\in S$ is labeled by the corresponding exponent $m_{st}$.

For a subset $T\subseteq S$ the \emph{special subgroup} $W_{T}$ of $W$ generated by $T$ is also a Coxeter group with the same exponents as $W$ (see \cite[Theorem 4.1.6]{Davis}). The system $(W,S)$ is \emph{irreducible} if its Coxeter diagram is connected. This is the case if and only if $W$ does not decompose as a non-trivial direct product of special subgroups.

Every element $\mathbf{w}\in W$ decomposes as a product $\mathbf{w}=s_{1}...s_{n}$ of generators $s_{1},...,s_{n}\in S$. The expression $s_1 ... s_n$ is called \emph{reduced} if it has minimal length. The \emph{word length} of $\mathbf{w}$, denoted by $\left|\mathbf{w}\right|$, is then defined to be the number of generators appearing in a reduced expression for $\mathbf{w}$, where $\left|e\right|:=0$. One says that $\mathbf{w}$ \emph{starts} (resp. \emph{ends}) with $\mathbf{v}\in W$ if $|\mathbf{v}^{-1}\mathbf{w}|=|\mathbf{w}|-|\mathbf{v}|$ (resp. $|\mathbf{w}\mathbf{v}^{-1}|=|\mathbf{w}|-|\mathbf{v}|$). In that case we write $\mathbf{v}\leq_{R}\mathbf{w}$ (resp. $\mathbf{v}\leq_{L}\mathbf{w}$). This defines a partial order on $W$ which is called the \emph{weak right} (resp. \emph{weak left}) \emph{Bruhat order} of $(W,S)$. It turns $W$ into a \emph{complete meet-semilattice} (see \cite[Theorem 3.2.1]{combinatorics}). To simplify the notation we will usually write $\leq$ instead of $\leq_{R}$.

In the right-angled case, cancellation of the form $s_{1}...s_{n}=s_{1}...\widehat{s_{i}}...\widehat{s_{j}}...s_{n}$ for $s_{1},...,s_{n}\in S$ implies that $s_{i}=s_{j}$ and that $s_{i}$ commutes with $s_{i+1}...s_{j-1}$ (this follows from \cite[Lemma 3.3.3]{Davis}). Here $s_{1}...s_{n}=s_{1}...\widehat{s_{i}}...\widehat{s_{j}}...s_{n}$ means that $s_i$ and $s_j$ are removed from the expression $s_{1}...s_{n}$. We will use this frequently without further mention. Another useful property is the following.

\begin{proposition}[{ \cite[Proposition 3.1.2 (vi)]{combinatorics}}] \label{help}
Let $(W,S)$ be a Coxeter system, $\mathbf{v},\mathbf{w}\in W$ and $s\in S$ with $s\leq\mathbf{v}$, $s\leq\mathbf{w}$. Then, $\mathbf{v}\leq\mathbf{w}$ if and only if $s\mathbf{v}\leq s\mathbf{w}$.
 \end{proposition}


\subsection{Topological boundaries of Coxeter groups}

In \cite{Mario2} topological boundaries and compactifications associated with connected rooted graphs were introduced and studied. These topological spaces are particularly useful in the case of (Cayley graphs of) Coxeter groups. For these, the spaces have been introduced earlier by Caprace and L\'{e}cureux in \cite{Caprace} and by Lam and Thomas in \cite{Lam} in a different setting, using different formalisms. The construction that we will follow in this paper coincides with the one in \cite{Mario2}, but we will restrict to the case of Cayley graphs of Coxeter groups. For more details and the general construction see \cite{Mario2}.\\

Let $(W,S)$ be a finite rank Coxeter system and denote by $K:=\text{Cay}(W,S)$ the \emph{Cayley graph} of $W$ with respect to the generating set $S$, i.e. the graph with vertex set $W$ and edge set $\{ (\mathbf{v},\mathbf{w})\in W\times W\mid\mathbf{v}^{-1}\mathbf{w}\in S\}$. The metric $d\text{: }W\times W\rightarrow\mathbb{R}_{\geq 0}$ defined by $d(\mathbf{v},\mathbf{w}):=|\mathbf{v}^{-1}\mathbf{w}|$ turns (the vertex set of) $K$ into a metric space. A \emph{geodesic path} $\alpha$ in $K$ is a (possibly infinite) sequence $\alpha_{0}\alpha_{1}...$ of vertices with $d(\alpha_{i},\alpha_{j})=\left|i-j\right|$ for all $i,j$. Without further comments we will often extend a finite geodesic path $\alpha_0 ... \alpha_n$ to an infinite path via $\alpha_0...\alpha_n\alpha_n\alpha_n...$ and still call it (finite) geodesic. For a geodesic path $\alpha$ and $\mathbf{w}\in W$ we write $\mathbf{w}\leq\alpha$ if $\mathbf{w}\leq\alpha_{i}$ for all large enough $i$ and we write $\mathbf{w}\nleq\alpha$ if $\mathbf{w}\nleq\alpha_{i}$ for all large enough $i$. Now, define an equivalence relation $\sim$ on the set of all infinite geodesic paths in $K$ via $\alpha\sim\beta$ if and only if for every $\mathbf{\mathbf{w}}\in W$ the implications $\mathbf{w}\leq\alpha\Leftrightarrow\mathbf{w}\leq\beta$ hold. Write $\partial(W,S)$ for the set of corresponding equivalence classes. This set is called the \emph{boundary} of $(W,S)$ and $\overline{(W,S)}:=W\cup\partial(W,S)$ is called the \emph{compactification} of $(W,S)$. The weak right Bruhat order naturally extends to a partial order $\leq$ on $\overline{(W,S)}$ (see \cite[Lemma 2.2]{Mario2}). We then equip $\overline{(W,S)}$ with the topology generated by the subbase of sets of the form
\begin{eqnarray}
\nonumber
\mathcal{U}_{\mathbf{w}}:=\left\{ z\in\overline{(W,S)}\mid\mathbf{w}\leq z\right\} \text{\: and \:}\mathcal{U}_{\mathbf{w}}^{c}:=\left\{ z\in\overline{(W,S)}\mid\mathbf{w}\nleq z\right\} \text{,}
\end{eqnarray}
where $\mathbf{w}\in W$. This turns $\partial(W,S)$ and $\overline{(W,S)}$ into metrizable compact spaces and $W$ naturally embeds as a dense discrete subset into $\overline{(W,S)}$. Further, the left action of $W$ on itself induces a (continuous) action $W\curvearrowright\overline{(W,S)}$ with $W.(\partial(W,S))=\partial(W,S)$. This action has some desirable properties, one of which will play a role in the characterization of the simplicity of right-angled Hecke C$^{\ast}$-algebras.

\begin{theorem}[{\cite[Theorem 3.20 and Proposition 3.26]{Mario2}}] \label{cancellation}
Let $(W,S)$ be a right-angled irreducible Coxeter system with $3\leq\#S<\infty$. Then the action $W\curvearrowright\partial(W,S)$ is a \emph{boundary action}, meaning that the following statements hold:

\begin{itemize}
\item \emph{Minimality:} For every $z\in\partial(W,S)$ the $W$-orbit $W.z:=\left\{ \mathbf{w}.z\mid\mathbf{w}\in W\right\}$  is dense in $\partial(W,S)$;
\item  \emph{Strong proximality:} For every probability measure $\nu\in\text{Prob}\left(\partial(W,S)\right)$ the weak-$\ast$ closure of the $W$-orbit $W.\nu$ contains a point mass $\delta_{z}\in\text{Prob}(\partial(W,S))$ for some $z\in\partial(W,S)$.
\end{itemize}
Further, the action is \emph{topologically free}, i.e. for every $\mathbf{w}\in W\setminus\left\{ e\right\}$  the set $(\partial(W,S))^{\mathbf{w}}:=\left\{ z\in\partial(W,S)\mid\mathbf{w}.z=z\right\}$  has no inner points.
\end{theorem}


\subsection{Multi-parameter Hecke algebras}

For a Coxeter system $(W,S)$ define $\mathbb{R}_{>0}^{(W,S)}$ to be the set of all multi-parameters $q=(q_s)_{s\in S} \in \mathbb{R}^S_{>0}$ for which $q_s = q_t$ for all $s,t\in S$ which are conjugate to each other. The sets $\mathbb{C}^{(W,S)}$ and $\left\{ -1,1\right\} ^{(W,S)}$ are defined in a similar way. For a tuple $q=(q_s)_{s\in S} \in \mathbb{R}_{>0}^{(W,S)}$, $s\in S$ and a reduced expression $\mathbf{w}=s_{1}...s_{n}$ of $\mathbf{w}\in W$ write
\begin{eqnarray}
\nonumber
q_{\mathbf{w}}:=q_{s_{1}}...q_{s_{n}}\text{\: and \:}p_{s}(q):=q_{s}^{-\frac{1}{2}}(q_{s}-1)\text{.}
\end{eqnarray}
Then $q_{\mathbf{w}}$ does not depend on the choice of the reduced expression for $\mathbf{w}$ (see \cite[Chapter 17.1]{Davis}). Following the notation in \cite{Raum} we also write $q_{s,\epsilon}:=\epsilon_{s}q_{s}^{\epsilon_{s}}$ and $q_{\mathbf{w},\epsilon}:=q_{s_{1},\epsilon}...q_{s_{n},\epsilon}$ for $q\in\mathbb{R}_{>0}^{(W,S)}$, $\epsilon\in\left\{ -1,1\right\} ^{(W,S)}$.\\

By \cite[Proposition 19.1.1]{Davis} for $q\in\mathbb{R}_{>0}^{(W,S)}$ there exists a unique (unital) $\ast$-algebra $\mathbb{C}_{q}\left[W\right]$ spanned by a linear basis $\{ T_{\mathbf{w}}^{(q)}\mid\mathbf{w}\in W\}$  such that for  $s\in S$, $\mathbf{w}\in W$ one has
\begin{eqnarray} \label{multiplication}
T_{s}^{\left(q\right)}T_{\mathbf{w}}^{\left(q\right)}=\begin{cases}
T_{s\mathbf{w}}^{\left(q\right)} & \text{, if }s \nleq \mathbf{w} \\
T_{s\mathbf{w}}^{\left(q\right)}+p_{s}(q)T_{\mathbf{w}}^{\left(q\right)} & \text{, if } s \leq \mathbf{w}
\end{cases}
\end{eqnarray}
and \begin{eqnarray} \nonumber (T_{\mathbf{w}}^{(q)})^{\ast}=T_{\mathbf{w}^{-1}}^{(q)} \text{.} \end{eqnarray}
This $\ast$-algebra is called the \emph{(Iwahori) Hecke algebra} of $(W,S)$ with parameter $q$. Here we use a different normalization of the generators than in \cite{Davis}, which coincides with the notation in \cite{Gar}, \cite{Caspers}, \cite{CSW}, \cite{Mario}, \cite{Raum} and \cite{Mario2}. The equality \eqref{multiplication} in particular implies that $T_{\mathbf{w}}^{(q)}=T_{s_{1}}^{(q)}...T_{s_{n}}^{(q)}$ for a reduced expression $\mathbf{w}=s_{1}...s_{n}$ of $\mathbf{w}\in W$. The $\ast$-algebra $\mathbb{C}_{q}\left[W\right]$ can be represented on $\ell^2(W)$ by bounded operators via 
\begin{eqnarray} \nonumber
T_{s}^{\left(q\right)}\delta_{\mathbf{w}}=\begin{cases}
\delta_{s\mathbf{w}} & \text{, if } s\nleq \mathbf{w}\\
\delta_{s\mathbf{w}}+p_{s}(q)\delta_{\mathbf{w}} & \text{, if } s\leq\mathbf{w}
\end{cases},
\end{eqnarray}
where $(\delta_{\mathbf{w}})_{\mathbf{w}\in W}$ denotes the canonical orthonormal basis of $\ell^{2}(W)$. This defines a faithful $\ast$-representation $\mathbb{C}_{q}\left[W\right]\hookrightarrow\mathcal{B}(\ell^{2}(W))$, therefore we will view $\mathbb{C}_{q}\left[W\right]$ as a $\ast$-subalgebra of $\mathcal{B}(\ell^{2}(W))$. The norm closure $C_{r,q}^{\ast}(W):=\overline{\mathbb{C}_{q}\left[W\right]}^{\left\Vert \cdot\right\Vert }$ is called the \emph{(reduced) Hecke C$^{\ast}$-algebra} and the weak closure $\mathcal{N}_{q}(W):=\overline{\mathbb{C}_{q}\left[W\right]}^{w.o.}$ is called the \emph{Hecke-von Neumann algebra}. Note that for $q_{s}=1$, $s\in S$,
\begin{eqnarray}
\nonumber
\mathbb{C}_{q}\left[W\right]=\mathbb{C}\left[W\right] \text{, } C_{r,q}^{\ast}(W)=C_{r}^{\ast}(W) \text{, } \mathcal{N}_{q}(W)=\mathcal{L}(W)
\end{eqnarray}
are the group algebra, reduced group C$^{\ast}$-algebra and group-von Neumann algebra of $W$. Further, for every $q\in\mathbb{R}_{>0}^{(W,S)}$ the vector state $x\mapsto\left\langle x\delta_{e},\delta_{e}\right\rangle$  restricts to a faithful tracial state $\tau_{q}$ on $C_{r,q}^{\ast}(W)$ and $\mathcal{N}_{q}(W)$ with $\tau_{q}(T_{\mathbf{w}}^{(q)})=0$ for all non-trivial $\mathbf{w}\in W$.\\

The following statement is well-known. A proof can be found in \cite[Proposition 4.7]{Mario}.

\begin{proposition} \label{isomorphism}
Let $(W,S)$ be a Coxeter system, $q=(q_{s})_{s\in S}\in\mathbb{R}_{>0}^{(W,S)}$ and $\epsilon=(\epsilon_{s})_{s\in S}\in\left\{ -1,1\right\} ^{(W,S)}$. Set  $q^{\prime}:=(q_{s}^{\epsilon_{s}})_{s\in S}$. Then $C_{r,q}^{\ast}(W)\cong C_{r,q^{\prime}}^{\ast}(W)$ via $T_{s}^{(q)}\mapsto\epsilon_{s}T_{s}^{(q^{\prime})}$.
\end{proposition}

The following decomposition follows from the universal property of the Hecke algebra and \cite[Lemma 1.1]{Mario}.

\begin{lemma} \label{isomorphism2}
Let $(W,S)$ be a Coxeter system which admits a non-trivial decomposition of the form $(W,S)=(W_{T}\times W_{T^{\prime}},T \sqcup T^{\prime})$. Set $q_{T}:=(q_{t})_{t\in T}$ and $q_{T^{\prime}}:=(q_{t})_{t\in T^{\prime}}$. Then for every $q\in\mathbb{R}_{>0}^{(W,S)}$ the corresponding Hecke algebra decomposes as an algebraic tensor product $\mathbb{C}_{q}\left[W\right]\cong\mathbb{C}_{q_{T}}\left[W_{T}\right]\odot\mathbb{C}_{q_{T^{\prime}}}\left[W_{T^{\prime}}\right]$ via
\begin{eqnarray}
\nonumber
T_{t}^{(q)}\mapsto \begin{cases}
T_{t}^{(q_{T})}\otimes1 & \text{, if } t\in T\\
1\otimes T_{t}^{(q_{T^{\prime}})} & \text{, if } t\in T^{\prime}
\end{cases}.
\end{eqnarray}
This induces C$^{\ast}$-algebraic and von Neumann-algebraic isomorphisms $C_{r,q}^{\ast}(W)\cong C_{r,q_{T}}^{\ast}(W_{T})\otimes C_{r,q_{T^{\prime}}}^{\ast}(W_{T^{\prime}})$ and $\mathcal{N}_{q}(W)\cong\mathcal{N}_{q_{T}}(W_{T})\overline{\otimes}\mathcal{N}_{q_{T^{\prime}}}(W_{T^{\prime}})$.
\end{lemma}

Proposition \ref{isomorphism} and Lemma \ref{isomorphism2} allow to restrict in the treatment of the question for the simplicity of Hecke C$^{\ast}$-algebras to irreducible Coxeter systems and multi-parameters $q=(q_{s})_{s\in S}\in\mathbb{R}_{>0}^{(W,S)}$ with $0<q_{s}\leq1$. Since simplicity of C$^{\ast}$-algebras is preserved by inductive limits, it further suffices to consider finite rank Coxeter groups, as the following lemma illustrates. It easily follows from \cite[Lemma 19.2.2]{Davis}.

\begin{lemma}\label{isomorphism3}
Let $(W,S)$ be a Coxeter system, $q=(q_{s})_{s\in S}\in\mathbb{R}_{>0}^{(W,S)}$, $T_0 \subseteq S$ finite and $\mathcal{S}:=\left\{ T \subseteq S\mid T \text{ finite with }T_0 \subseteq T \right\}$. For $T \in \mathcal{S}$ set $q_{T}:=(q_t)_{t \in T}$.  Then 
\begin{eqnarray}
\nonumber
\left\{ (C_{r,q_{T}}^{\ast}(W_{T}),\phi_{T,T^{\prime}})\mid T,T^{\prime}\in\mathcal{S}\text{ with }T\subseteq T^{\prime}\right\}
\end{eqnarray}
with $\phi_{T,T^{\prime}}(T_{t}^{(q_{T})}):=T_{t}^{(q_{T^{\prime}})}$ for $t\in T$ defines an inductive system with $C_{r,q}^{\ast}(W)\cong\underrightarrow{\lim}C_{r,q_T}^{\ast}(W_{T})$.
\end{lemma}

\vspace{1mm}


\section{The C$^{\ast}$-algebras $\mathcal{D}(W,S)$ and $\mathfrak{A}(W)$} \label{2}

The aim of this section is to recall the construction of the C$^{\ast}$-algebras $\mathcal{D}(W,S)$ and $\mathfrak{A}(W)$ associated with a given Coxeter system $(W,S)$ which appears in \cite[Section 4]{Mario2}. We will further prove a number of technical statements which will play a role in the later sections.\\

Let $(W,S)$ be a finite rank Coxeter system and define for $\mathbf{w}\in W$, $P_{\mathbf{w}}\in\ell^{\infty}(W)\subseteq\mathcal{B}(\ell^{2}(W))$ to be the orthogonal projection onto $\overline{\text{Span}\left\{ \delta_{\mathbf{v}}\mid\mathbf{v}\in W\text{ with }\mathbf{w}\leq\mathbf{v}\right\} }\subseteq\ell^{2}(W)$. Note that $P_{e}=1$.

\begin{remark} \label{remark}
Recall that $W$ equipped with the weak right Bruhat order defines a complete meet-semilattice. If existent, denote the corresponding join of two elements $\mathbf{v},\mathbf{w}\in W$ by $\mathbf{v}\vee\mathbf{w}$. We then have $P_{\mathbf{v}}P_{\mathbf{w}}=P_{\mathbf{v}\vee\mathbf{w}}$ for all $\mathbf{v},\mathbf{w}\in W$ where we assume that $P_{\mathbf{v}\vee\mathbf{w}}=0$ if the join $\mathbf{v}\vee\mathbf{w}$ does not exist. In particular, the equalities $P_{s}P_{t}=0$ (i.e. $P_{s}$ and $P_{t}$ are orthogonal to each other) if $m_{st}=\infty$ and $P_{s}P_{t}=P_{st}$ if $m_{st}=2$ hold (this follows for instance from \cite[Lemma 4.3.3]{Davis}).
\end{remark}

Denote the quotient map of $\mathcal{B}(\ell^{2}(W))$ onto $\mathcal{B}(\ell^{2}(W))/\mathcal{K}$ by $\pi$ where $\mathcal{K}:=\mathcal{K}(\ell^{2}(W))$ is the ideal of compact operators in $\mathcal{B}(\ell^{2}(W))$ and write $\tilde{P}_{\mathbf{w}}:=\pi(P_{\mathbf{w}})$ for $\mathbf{w}\in W$. It was shown in \cite[Proposition 2.6]{Mario2} that the commutative C$^{\ast}$-algebra $\mathcal{D}(W,S)$ generated by all projections $P_{\mathbf{w}}$, $\mathbf{w}\in W$ identifies with $C(\overline{(W,S)})$ via $P_{\mathbf{w}}\mapsto\chi_{\mathcal{U}_{\mathbf{w}}}$.
Similarly, by \cite[Lemma 2.3 and Proposition 2.6]{Mario2}, $\pi(\mathcal{D}(W,S))\cong C(\partial(W,S))$ via $\tilde{P}_{\mathbf{w}}\mapsto\chi_{\mathcal{U}_{\mathbf{w}}\cap\partial(W,S)}$. Further, let $\mathfrak{A}(W)$ be the C$^{\ast}$-subalgebra of $\mathcal{B}(\ell^{2}(W))$ generated by the reduced group C$^{\ast}$-algebra $C_{r}^{\ast}(W)$ and $\mathcal{D}(W,S)$. Since for $q=(q_{s})_{s\in S}\in\mathbb{R}_{>0}^{(W,S)}$ and $s\in S$ the operator $T_{s}^{(q)}$ decomposes as $T_{s}^{(1)}+p_{s}(q)P_{s}$, we have an inclusion of the corresponding Hecke C$^{\ast}$-algebra $C_{r,q}^{\ast}(W)\subseteq\mathfrak{A}(W)$. In fact, $\mathfrak{A}(W)$ is the smallest C$^{\ast}$-subalgebra of $\mathcal{B}(\ell^{2}(W))$ containing all Hecke C$^{\ast}$-algebras of $(W,S)$. It naturally identifies with the reduced crossed product C$^{\ast}$-algebra of the action $W\curvearrowright\overline{(W,S)}$ via
\begin{eqnarray}
\nonumber
\iota\text{: }\mathfrak{A}(W)\cong C(\overline{(W,S)})\rtimes_{r}W\text{, }P_{\mathbf{w}}\mapsto\chi_{\mathcal{U}_{\mathbf{w}}}\text{ and }T_{\mathbf{w}}^{(1)}\mapsto\lambda_{\mathbf{w}}\text{,}
\end{eqnarray}
where $\lambda$ denotes the left-regular representation of $W$. In a similar way $\pi(\mathfrak{A}(W))$ identifies with $C(\partial(W,S)) \rtimes_{r} W$ via
\begin{eqnarray}
\nonumber
\kappa\text{: }\pi(\mathfrak{A}(W))\cong C(\partial(W,S)) \rtimes_{r} W\text{, }\tilde{P}_{\mathbf{w}}\mapsto\chi_{\mathcal{U}_{\mathbf{w}}\cap\partial(W,S)}\text{ and }\pi(T_{\mathbf{w}}^{(1)})\mapsto\lambda_{\mathbf{w}}\text{.}
\end{eqnarray}
These maps are $W$-equivariant with respect to the action of $W$ on $\mathfrak{A}(W)$ defined by $\mathbf{w}.x:=T_{\mathbf{w}}^{(1)}xT_{\mathbf{w}^{-1}}^{(1)}$ for $\mathbf{w}\in W$, $x\in\mathfrak{A}(W)$ and the action of $W$ on $\pi(\mathfrak{A}(W))$ defined by $\mathbf{w}.x:=$ \: $\pi(T_{\mathbf{w}}^{(1)})x\pi(T_{\mathbf{w}^{-1}}^{(1)})$ for $\mathbf{w}\in W$, $x\in\pi(\mathfrak{A}(W))$. Denote the \emph{region of convergence} of the multi-variate growth series $W(z):=\sum_{\mathbf{w}\in W}z_{\mathbf{w}}$ by
\begin{eqnarray}
\nonumber
\mathcal{R}:=\{ z\in\mathbb{C}^{(W,S)}\mid W(z)\text{ converges}\},
\end{eqnarray}
set
\begin{eqnarray}
\nonumber
\mathcal{R}^{\prime}:=\{ (q_{s}^{\epsilon_{s}})_{s\in S}\mid q\in\mathcal{R}\cap\mathbb{R}_{>0}^{(W,S)}\text{, }\epsilon\in\left\{ -1,1\right\} ^{(W,S)}\}
\end{eqnarray}
and let $\overline{\mathcal{R}^{\prime}}$ be the closure of $\mathcal{R}^{\prime}$ in $\mathbb{R}_{>0}^{(W,S)}$. By \cite[Corollary 4.4]{Mario2}, for $q\in\mathbb{R}_{>0}^{(W,S)}\setminus\mathcal{R}^{\prime}$ the restriction of $\kappa\circ\pi$ to $C_{r,q}^{\ast}(W)$ factors to an embedding of $C_{r,q}^{\ast}(W)$ into $C(\partial(W,S)) \rtimes_{r} W$. We will therefore often view $C_{r,q}^{\ast}(W)$ with $q\in\mathbb{R}_{>0}^{(W,S)}\setminus\mathcal{R}^{\prime}$ as a C$^{\ast}$-subalgebra of $C(\partial(W,S)) \rtimes_{r} W$ and of $\pi(\mathfrak{A}(W))$.


\subsection{Elementary properties of the action $W\curvearrowright\mathcal{D}(W,S)$}

Let us proceed with some technical statements which will play a role in the following sections.

\begin{proposition} \label{action}
Let $(W,S)$ be a right-angled Coxeter system and $\mathbf{w}\in W$, $s\in S$. Then the following equalities hold:

\begin{enumerate}
\item $s.P_{\mathbf{w}}=P_{s\mathbf{w}}$ if $\mathbf{w}\notin C_{W}(s)$;
\item $s.P_{\mathbf{w}}=P_{s\mathbf{w}}-P_{\mathbf{w}}$ if $\mathbf{w}\in C_{W}(s)$  and $s\leq\mathbf{w}$;
\item $s.P_{\mathbf{w}}=P_{\mathbf{w}}$ if $\mathbf{w}\in C_{W}(s)$ and $s\nleq\mathbf{w}$.
\end{enumerate}
Here $C_{W}(s):=\left\{ \mathbf{v}\in W\mid s\mathbf{v}=\mathbf{v}s\right\}$  denotes the centralizer of $s$ in $W$.
\end{proposition}

\begin{proof}
First observe that by Proposition \ref{help} for all $s\in S$ and $\mathbf{v},\mathbf{w}\in W$ with $s\leq\mathbf{w}$, $s\nleq\mathbf{v}$ or $s\nleq\mathbf{w}$, $s\leq\mathbf{v}$,
 \begin{eqnarray} \label{equality}
(s.P_{\mathbf{w}})\delta_{\mathbf{v}}=T_{s}^{(1)}P_{\mathbf{w}}\delta_{s\mathbf{v}}=
\begin{cases}
\delta_{\mathbf{v}} & \text{, if } \mathbf{w} \leq s\mathbf{v}\\
0 & \text{, if } \mathbf{w} \nleq s\mathbf{v}
\end{cases}
=
\begin{cases}
\delta_{\mathbf{v}} & \text{, if } s\mathbf{w} \leq \mathbf{v}\\
0 & \text{, if } s\mathbf{w} \nleq \mathbf{v}
\end{cases}
=P_{s\mathbf{w}}\delta_{\mathbf{v}}\text{.}
\end{eqnarray}
We will cover the remaining cases in the following.\\

\emph{(1)}: Assume that $\mathbf{w}\notin C_{W}(s)$. If $s\leq\mathbf{w}$ and $s\leq\mathbf{v}$, then $\mathbf{w}\nleq s\mathbf{v}$ and $s\mathbf{w}\nleq\mathbf{v}$. Indeed, if we assume that $\mathbf{w}\leq s\mathbf{v}$, then $s\leq s\mathbf{v}$ in contradiction to $s\nleq s\mathbf{v}$. Further, if we assume that $s\mathbf{w}\leq\mathbf{v}$, then there exists $\mathbf{u}\in W$ with $\mathbf{v}=(s\mathbf{w})\mathbf{u}$ and $\left|\mathbf{v}\right|=\left|s\mathbf{w}\right|+\left|\mathbf{u}\right|$. Since $(W,S)$ is right-angled $s\leq\mathbf{v}$ implies that $s\leq\mathbf{u}$ and $s\mathbf{w}\in C_{W}(s)$. But then $\mathbf{w}\in C_{W}(s)$ in contradiction to the assumption $\mathbf{w}\notin C_{W}(s)$. We get that $(s.P_{\mathbf{w}})\delta_{\mathbf{v}}=0=P_{s\mathbf{w}}\delta_{\mathbf{v}}$.

If $s\nleq\mathbf{w}$ and $s\nleq\mathbf{v}$, then one obtains in the same way $\mathbf{w}\nleq s\mathbf{v}$ and $s\mathbf{w}\nleq\mathbf{v}$ which implies $(s.P_{\mathbf{w}})\delta_{\mathbf{v}}=0=P_{s\mathbf{w}}\delta_{\mathbf{v}}$. With \eqref{equality} this covers all possible cases. Hence, $s.P_{\mathbf{w}}=P_{s\mathbf{w}}$.

\emph{(2)}: Assume that $\mathbf{w}\in C_{W}(s)$ and $s\leq\mathbf{w}$. If $s\nleq\mathbf{v}$, then \eqref{equality} implies $(s.P_{\mathbf{w}})\delta_{\mathbf{v}}=P_{s\mathbf{w}}\delta_{\mathbf{v}}$ and hence $(s.P_{\mathbf{w}})(1-P_{s})=P_{s\mathbf{w}}(1-P_{s})$. If $s\leq\mathbf{v}$, then $\mathbf{w}\nleq s\mathbf{v}$ implies $(s.P_{\mathbf{w}})\delta_{\mathbf{v}}=0$ and hence $(s.P_{\mathbf{w}})P_{s}=0$. Combined this leads to
\begin{eqnarray}
\nonumber
s.P_{\mathbf{w}}=(s.P_{\mathbf{w}})(1-P_{s})+(s.P_{\mathbf{w}})P_{s}=P_{s\mathbf{w}}(1-P_{s})=P_{s\mathbf{w}}-P_{s\mathbf{w}\vee s}=P_{s\mathbf{w}}-P_{\mathbf{w}}\text{.}
\end{eqnarray}

\emph{(3)}: Assume that $\mathbf{w}\in C_{W}(s)$ and $s\nleq\mathbf{w}$. If $s\leq\mathbf{v}$, then $(s.P_{\mathbf{w}})\delta_{\mathbf{v}}=P_{s\mathbf{w}}\delta_{\mathbf{v}}=P_{\mathbf{w}\vee s}\delta_{\mathbf{v}}=P_{\mathbf{w}}P_{s}\delta_{\mathbf{v}}=P_{\mathbf{w}}\delta_{\mathbf{v}}$ by \eqref{equality}. So consider the case where $s\nleq\mathbf{v}$. If $\mathbf{w}\leq\mathbf{v}$, then $\mathbf{v}=\mathbf{w}\mathbf{u}$ for some $\mathbf{u}\in W$ with $\left|\mathbf{v}\right|=\left|\mathbf{w}\right|+\left|\mathbf{u}\right|$ and $s\nleq\mathbf{u}$. Hence, $s\mathbf{v}=\mathbf{w}(s\mathbf{u})\geq\mathbf{w}$. Conversely, if $\mathbf{w}\leq s\mathbf{v}$, then $s\mathbf{v}=\mathbf{w}\mathbf{u}$ for some $\mathbf{u}\in W$ with $\left|s\mathbf{v}\right|=\left|\mathbf{w}\right|+\left|\mathbf{u}\right|$ and $s\leq\mathbf{u}$. We get that $\mathbf{v}=s(s\mathbf{v})=\mathbf{w}(s\mathbf{u})\geq\mathbf{w}$. Together this gives $(s.P_{\mathbf{w}})\delta_{\mathbf{v}}=P_{\mathbf{w}}\delta_{\mathbf{v}}$, i.e. $s.P_{\mathbf{w}}=P_{\mathbf{w}}$ as claimed.
\end{proof}

\begin{remark} \label{remark2}
Let $(W,S)$ be a right-angled Coxeter system, $q\in\mathbb{R}_{>0}^{(W,S)}$ and $s\in S$, $\mathbf{w}\in W$. Recall that $T_{s}^{(q)}=T_{s}^{(1)}+p_{s}(q)P_{s}$ for $q\in\mathbb{R}_{>0}^{(W,S)}$. In combination with Remark \ref{remark} the Proposition \ref{action} leads to a description of the conjugation of the generating projections in $\mathcal{D}(W,S)$ with the Hecke operators $T_{s}^{(q)}$, $s\in S$. In particular, for $s\in S$, $\mathbf{w}\in W$ with $\mathbf{w}\notin C_{W}(s)$ and $s\nleq\mathbf{w}$ the identities
\begin{eqnarray}
\nonumber
T_{s}^{(q)}(1-P_{s})T_{s}^{(q)}=T_{s}^{(1)}(1-P_{s})T_{s}^{(1)}=P_{s}
\end{eqnarray}
and
\begin{eqnarray}
\nonumber
T_{s}^{(q)}P_{\mathbf{w}}T_{s}^{(q)}=T_{s}^{(1)}P_{\mathbf{w}}T_{s}^{(1)}=P_{s\mathbf{w}}
\end{eqnarray}
hold.
\end{remark}


\subsection{Paths in the Coxeter diagram of right-angled Coxeter groups}

To simplify the statements and proofs of later sections, we introduce the following notion which already implicitly appears in the proof of \cite[Theorem 3.20]{Mario2}.

\begin{definition}
Let $(W,S)$ be a right-angled finite rank Coxeter system. A \emph{path $s_{1}...s_{n}\in W$ in the Coxeter diagram of $(W,S)$} is a product of generators $s_{1},...,s_{n}\in S$ with $m_{s_{i}s_{i+1}}=\infty$ for $i=1,...n-1$. We say that the path is \emph{closed} if $m_{s_{1}s_{n}}=\infty$ and that the path \emph{covers the whole graph} if $\left\{ s_{1},...,s_{n}\right\} =S$.
\end{definition}

\begin{remark}
Let $(W,S)$ be a right-angled finite rank Coxeter system. For a closed path $\mathbf{g}:=s_{1}...s_{n}\in W$ in the Coxeter diagram of $(W,S)$ that covers the whole graph we have that $\left|s\mathbf{g}\right|>\left|\mathbf{g}\right|$ for every $s\in S\setminus\left\{ s_{1}\right\}$  and $C_{W}(\mathbf{g})=\left\{ \mathbf{g}^{i}\mid i\in\mathbb{Z}\right\}$. In particular, $\left|\mathbf{g}^{n}\right|=\left|n\right|\left|\mathbf{g}\right|$ for every $n\in\mathbb{Z}$. These facts have been crucial in the proof of \cite[Theorem 3.20]{Mario2}.
\end{remark}

In the single-parameter case the following lemma appears in \cite[Lemma 2.7]{Caspers}. The proof presented there translates verbatim to the multi-parameter case. Therefore we omit it here.

\begin{lemma}[{\cite[Lemma 2.7]{Caspers}}] \label{Martijn}
Let $(W,S)$ be a right-angled Coxeter system. Denote the set of subsets of $S$ whose elements pairwise commute (including the empty set) by $\text{Cliq}$ and write $P_{\Gamma}:=\prod_{s\in\Gamma}P_{s}$ for $\Gamma\in\text{Cliq}$. Further, for $\mathbf{w}\in W$ let $A_{\mathbf{w}}$ be the set of triples $\left(\mathbf{w}^{\prime},\Gamma,\mathbf{w}^{\prime\prime}\right)$ with $\mathbf{w}^{\prime},\mathbf{w}^{\prime\prime}\in W$ and $\Gamma\in\text{Cliq}$ such that $\mathbf{w}=\mathbf{w}^{\prime}\left(\prod_{s\in\Gamma}s\right)\mathbf{w}^{\prime\prime}$, $|\mathbf{w}|=|\mathbf{w}^{\prime}|+|\prod_{s\in\Gamma}s|+|\mathbf{w}^{\prime\prime}|$ and $|\mathbf{w}^{\prime}t|>|\mathbf{w}^{\prime}|$ for all $t\in S$ with $m_{st}=2$ for all $s\in\Gamma$. Then the operator $T_{\mathbf{w}}^{(q)}$ decomposes as
\begin{eqnarray}
\nonumber
T_{\mathbf{w}}^{(q)}=\sum_{\left(\mathbf{w}^{\prime},\Gamma,\mathbf{w}^{\prime\prime}\right)\in A_{\mathbf{w}}}\left(\prod_{s\in\Gamma}p_{s}(q)\right)T_{\mathbf{w}^{\prime}}^{(1)}P_{\Gamma}T_{\mathbf{w}^{\prime\prime}}^{(1)}\text{.}
\end{eqnarray}
\end{lemma}

\begin{corollary} \label{decomposition}
Let $(W,S)$ be a right-angled Coxeter system, $q=(q_{s})_{s\in S}\in\mathbb{R}_{>0}^{(W,S)}$, $l\in\mathbb{N}$ and let $\mathbf{g}:=s_{1}...s_{n}\in W$ be a closed path in the Coxeter diagram of $(W,S)$. Then there exists an operator $x\in\mathfrak{A}(W)$ such that $T_{\mathbf{g}^{l}}^{(q)}$ decomposes as $T_{\mathbf{g}^{l}}^{(q)}=T_{\mathbf{g}^{l}}^{(1)}+P_{s_{1}}x$.
\end{corollary}

\begin{proof}
Write $t_{1}...t_{m}$ for the reduced expression $\mathbf{g}^{l}=(s_{1}...s_{n})(s_{1}...s_{n})...(s_{1}...s_{n})$ of $\mathbf{g}^{l}$ where $m=nl$. By Lemma \ref{Martijn} the operator $T_{\mathbf{g}^{l}}^{(q)}$ decomposes as 
\begin{eqnarray}
\nonumber
T_{\mathbf{g}^{l}}^{(q)}=T_{\mathbf{g}^{l}}^{(1)}+\sum_{i=1}^{m}p_{t_{i}}(q)T_{t_{1}...t_{i-1}}^{(1)}P_{t_{i}}T_{t_{i+1}...t_{m}}^{(1)}\text{,}
\end{eqnarray}
where the first summand corresponds to the triple $(e,\emptyset,\mathbf{g}^{l})\in A_{\mathbf{g}^{l}}$
and the other summands correspond to triples $(t_{1}...t_{i-1},t_{i},t_{i+1}...t_{m})\in A_{\mathbf{g}^{l}}$
with $i=1,...,m$ (since $\mathbf{g}$ is a closed path in the Coxeter
diagram of $(W,S)$ all triples in $A_{\mathbf{g}^{l}}$ are of these
forms). Using the description in Proposition \ref{action} we get that
\begin{eqnarray}
\nonumber
T_{\mathbf{g}^{l}}^{(q)}=T_{\mathbf{g}^{l}}^{(1)}+\sum_{i=1}^{m} p_{t_i}(q) T_{t_{1}...t_{i-2}}^{(1)}P_{t_{i-1}t_{i}}T_{t_{i-1}t_{i+1}...t_{m}}^{(1)}=...=T_{\mathbf{g}^{l}}^{(1)}+\sum_{i=1}^{m} p_{t_i}(q) P_{t_{1}...t_{i}}T_{t_{1}...\widehat{t_{i}}...t_{m}}^{(1)}\text{,}
\end{eqnarray}
so the claim follows by setting $x:=\sum_{i=1}^{m} p_{t_i}(q) P_{t_{1}...t_{i}}T_{t_{1}...\widehat{t_{i}}...t_{n}}^{(1)}$.
\end{proof}

\begin{lemma} \label{step 0}
Let $(W,S)$ be an irreducible right-angled, finite rank Coxeter system, let $\mathbf{g}:=s_{1}...s_{n}\in W$ be a path in the Coxeter diagram of $(W,S)$ that covers the whole graph and let $q=(q_{s})_{s\in S}\in\mathbb{R}_{>0}^{(W,S)}$. Then the series $\sum_{\mathbf{w}\in W}q_{\mathbf{w}}$ converges if and only if the series $\sum_{\mathbf{w}\in W\text{: }\mathbf{g}\leq\mathbf{w}^{-1}}q_{\mathbf{w}}$ converges.
\end{lemma}

\begin{proof}
Since all summands of the series are positive it is clear that the convergence of $\sum_{\mathbf{w}\in W}q_{\mathbf{w}}$ implies the convergence of $\sum_{\mathbf{w}\in W\text{: }\mathbf{g}\leq\mathbf{w}^{-1}}q_{\mathbf{w}}$. So assume that $\sum_{\mathbf{w}\in W\text{: }\mathbf{g}\leq\mathbf{w}^{-1}}q_{\mathbf{w}}$ converges. For every $i,j\in\mathbb{N}$ with $i<j$ we have that
\begin{eqnarray} \nonumber
\left|\sum_{\mathbf{w}\in W\text{: }\left|\mathbf{w}\right|\leq i}q_{\mathbf{w}}-\sum_{\mathbf{w}\in W\text{: }\left|\mathbf{w}\right|\leq j}q_{\mathbf{w}}\right| &=& \sum_{\mathbf{w}\in W\text{: }i<\left|\mathbf{w}\right|\leq j}q_{\mathbf{w}}\\
\nonumber
&=& \sum_{\mathbf{w}\in W\text{: }i<\left|\mathbf{w}\right|\le j\text{, }s_{n}\leq\mathbf{w}^{-1}}q_{\mathbf{w}}+\sum_{\mathbf{w}\in W\text{: }i<\left|\mathbf{w}\right|\leq j\text{, }s_{n}\nleq\mathbf{w}^{-1}}q_{\mathbf{w}}\\
\nonumber
&=& \sum_{\mathbf{w}\in W\text{: }i-1<\left|\mathbf{w}\right|\leq j-1\text{, }s_{n}\nleq\mathbf{w}^{-1}}q_{s_{n}}q_{\mathbf{w}}+\sum_{\mathbf{w}\in W\text{: }i<\left|\mathbf{w}\right|\leq j\text{, }s_{n}\nleq\mathbf{w}^{-1}}q_{\mathbf{w}}\\
\nonumber
&\leq& (1+q_{s_{n}})\sum_{\mathbf{w}\in W\text{: }i-1<\left|\mathbf{w}\right|\leq j\text{, }s_{n}\nleq\mathbf{w}^{-1}}q_{\mathbf{w}}\\
\nonumber
&=& \frac{1+q_{s_{n}}}{q_{\mathbf{g}^{-1}}}\sum_{\mathbf{w}\in W\text{: }i-1<\left|\mathbf{w}\right|\leq j\text{, }s_{n}\nleq\mathbf{w}^{-1}}q_{\mathbf{g}^{-1}}q_{\mathbf{w}}\\
\nonumber
&=& \frac{1+q_{s_{n}}}{q_{\mathbf{g}}}\sum_{\mathbf{w}\in W\text{: }i-1<\left|\mathbf{g}^{-1}\mathbf{w}^{-1}\right|\leq j\text{, }\mathbf{g}\leq\mathbf{w}^{-1}}q_{\mathbf{w}}\\
\nonumber
&=& \frac{1+q_{s_{n}}}{q_{\mathbf{g}}}\sum_{\mathbf{w}\in W\text{: }i-1+n<\left|\mathbf{w}\right|\leq j+n\text{, }\mathbf{g}\leq\mathbf{w}^{-1}}q_{\mathbf{w}}\text{,}
\end{eqnarray}
where the fifth equality follows from the fact that
\[
\{\mathbf{w}\mathbf{g}^{-1}\mid\mathbf{w}\in W\text{ with }s_{n}\nleq\mathbf{w}^{-1}\}=\{\mathbf{w}\in W\mid\mathbf{g}\leq\mathbf{w}^{-1}\}
\]
since $\mathbf{g}=s_{1}...s_{n}$ is a path in the Coxeter diagram
of $(W,S)$. That implies that the sequence of partial sums of $\sum_{\mathbf{w}\in W}q_{\mathbf{w}}$ is a Cauchy sequence and hence that the series converges.
\end{proof}


\subsection{States on $\mathfrak{A}(W)$ and $\pi(\mathfrak{A}(W))$}

Our solution of the simplicity question for right-angled Hecke C$^{\ast}$-algebras is inspired by Haagerup's approach to the unique trace property of group C$^{\ast}$-algebras in \cite{Haagerup}. The translation of the techniques into the deformed setting requires the study of states on the C$^{\ast}$-algebras $\mathfrak{A}(W)$ and $\pi(\mathfrak{A}(W))$.

\begin{lemma} \label{operator}
Let $(W,S)$ be a right-angled, finite rank Coxeter system. For every $\mathbf{u}\in W$ and $0<q<1$ the operator $\mathbf{Q}_{q}^{\mathbf{u}}$ on $\ell^{2}(W)$ defined by
\begin{eqnarray}
\nonumber
\mathbf{Q}_{q}^{\mathbf{u}}:=\sum_{l=\left|\mathbf{u}\right|}^{\infty} \; \sum_{\mathbf{w}\in W\text{: }\left|\mathbf{w}\right|=l,\mathbf{u}\leq\mathbf{w}^{-1}}q^{l}P_{\mathbf{w}}
\end{eqnarray}
exists (where the limit is taken with respect to the operator norm) and is contained in the C$^{\ast}$-algebra $\mathcal{D}(W,S)\subseteq\mathcal{B}(\ell^{2}(W))$.
\end{lemma}

\begin{proof}
It suffices to show that the sequence $(\mathbf{Q}_{q,i}^{\mathbf{u}})_{i\geq\left|\mathbf{u}\right|}$ with
\begin{eqnarray}
\nonumber
\mathbf{Q}_{q,i}^{\mathbf{u}}:=\sum_{l=\left|\mathbf{u}\right|}^{i} \; \sum_{\mathbf{w}\in W\text{: }\left|\mathbf{w}\right|=l,\mathbf{u}\leq\mathbf{w}^{-1}}q^{l}P_{\mathbf{w}}\in\mathcal{D}(W,S)
\end{eqnarray}
is a Cauchy sequence. For $\mathbf{v}\in W$ and $l\in\mathbb{N}$ set $\kappa_{\mathbf{v}}(l):=\#\left\{ \mathbf{w}\in W\mid\mathbf{w}\leq\mathbf{v}\text{ and }\left|\mathbf{w}\right|=l\right\}$. It has been shown in \cite[Lemma 4.4]{Caspers} that $\kappa_{\mathbf{v}}(l)\leq Cl^{\#S-2}$ for some constant $C>0$. Using this in the third line of the following inequalities, we get that for $i<j$ with $i\geq\left|\mathbf{u}\right|$ and $\xi\in\ell^{2}(W)$,
\begin{eqnarray}
\nonumber
\left\Vert (\mathbf{Q}_{q,j}^{\mathbf{u}}-\mathbf{Q}_{q,i}^{\mathbf{u}})\xi\right\Vert _{2}&=& \left\Vert \sum_{\mathbf{v}\in W}\left(\sum_{l=i+1}^{j} \; \sum_{\mathbf{w}\in W\text{: }\left|\mathbf{w}\right|=l,\mathbf{w}\leq\mathbf{v},\mathbf{u}\leq\mathbf{w}^{-1}}q^{l}\right)\xi(\mathbf{v})\delta_{\mathbf{v}}\right\Vert _{2}\\
\nonumber
&\leq& \sqrt{\sum_{\mathbf{v}\in W}\left(\sum_{l=i+1}^{j}q^{l}\kappa_{\mathbf{v}}(l)\right)^{2}\left|\xi(\mathbf{v})\right|^{2}}\\
\nonumber
&\leq& C\left(\sum_{l=i+1}^{j}q^{l}l^{\#S-2}\right)\left\Vert \xi\right\Vert _{2}\text{.}
\end{eqnarray}
For $0<q<1$ the series $\sum_{l=1}^{\infty}q^{l}l^{\#S-2}$ converges. This implies the claim.
\end{proof}

The following proposition will play a crucial role in Subsection \ref{3} and \ref{simplicity}. Recall that $\overline{\mathcal{R}^{\prime}}$ is the closure of $\mathcal{R}^{\prime}:=\left\{ (q_{s}^{\epsilon_{s}})_{s\in S}\mid q\in\mathcal{R}\cap\mathbb{R}_{>0}^{(W,S)}\text{, }\epsilon\in\left\{ -1,1\right\} ^{(W,S)}\right\}$  in $\mathbb{R}_{>0}^{(W,S)}$, where $\mathcal{R}$ denotes the region of convergence of the growth series of $(W,S)$.

\begin{proposition} \label{step 1}
Let $(W,S)$ be a right-angled, irreducible, finite rank Coxeter system, $q=(q_{s})_{s\in S}\in\mathbb{R}_{>0}^{(W,S)}\setminus\overline{\mathcal{R}^{\prime}}$ and let $\mathbf{g}:=s_{1}...s_{n}\in W$ be a path in the Coxeter diagram of $(W,S)$ that covers the whole graph. Then, for every state $\phi$ on $\mathfrak{A}(W)$ there exists a sequence $(\mathbf{w}_{i})_{i\in\mathbb{N}}\subseteq W$ of group elements with increasing word length such that $\mathbf{g}\leq\mathbf{w}_{i}^{-1}$ for all $i\in\mathbb{N}$ and $q_{\mathbf{w}_{i}}^{-1}\phi(P_{\mathbf{w}_{i}})\rightarrow0$.

The same statement holds, if one replaces $\mathfrak{A}(W)$ by $\pi(\mathfrak{A}(W))$ and $P_{\mathbf{w}_{i}}$ by $\tilde{P}_{\mathbf{w}_{i}}$.
\end{proposition}

\begin{proof}
The set $\mathbb{R}_{>0}^{(W,S)}\setminus\overline{\mathcal{R}^{\prime}}$ is open in $\mathbb{R}_{>0}^{(W,S)}$, so there exist positive real numbers $q^{\prime},\lambda\in\left(0,1\right)$ such that $q^{\prime}q:=(q^{\prime}q_{s})_{s\in S}\in\mathbb{R}_{>0}^{(W,S)}\setminus\overline{\mathcal{R}^{\prime}}$ and $\lambda q^{\prime}q:=(\lambda q^{\prime}q_{s})_{s\in S}\in\mathbb{R}_{>0}^{(W,S)}\setminus\overline{\mathcal{R}^{\prime}}$. In particular, Lemma \ref{step 0} implies that the series $\sum_{\mathbf{g}\leq\mathbf{w}^{-1}}\lambda^{\left|\mathbf{w}\right|}(q^{\prime}q)_{\mathbf{w}}$ diverges. By the root test criterium for convergence,
\begin{eqnarray}
\nonumber
\limsup_{l}\left(\sum_{\mathbf{w}\in W\text{: }\left|\mathbf{w}\right|=l,\mathbf{g}\leq\mathbf{w}^{-1}}\lambda^{l}(q^{\prime}q)_{\mathbf{w}}\right)^{1/l}\geq1
\end{eqnarray}
and hence
\begin{eqnarray}
\nonumber
\limsup_{l}\left(\sum_{\mathbf{w}\in W\text{: }\left|\mathbf{w}\right|=l,\mathbf{g}\leq\mathbf{w}^{-1}}(q^{\prime}q)_{\mathbf{w}}\right)^{1/l}>1\text{.}
\end{eqnarray}
One can thus find a strictly increasing sequence $\left(l_{i}\right)_{i\in\mathbb{N}}\subseteq\mathbb{N}$ and a constant $C>1$ such that for all $i\in\mathbb{N}$,
\begin{eqnarray} \label{*}
\left(\sum_{\mathbf{w}\in W\text{: }\left|\mathbf{w}\right|=l_{i},\mathbf{g}\leq\mathbf{w}^{-1}}(q^{\prime}q)_{\mathbf{w}}\right)^{1/l_{i}}\geq C\text{.}
\end{eqnarray}
For $\mathbf{w}\in W$ define the set
\begin{eqnarray}
\nonumber
C_{\mathbf{w}}:=\left\{ \mathbf{v}\in W\mid\mathbf{g}\leq\mathbf{v}^{-1}\text{ and }z_{\mathbf{v}}=z_{\mathbf{w}}\text{ for all }z=(z_{s})_{s\in S}\in\mathbb{C}^{(W,S)}\right\}
\end{eqnarray}
and note that the elements in $C_{\mathbf{w}}$ all have the same length. Choose for every $i\in\mathbb{N}$ an element $\mathbf{w}_{i}\in W$ with $\#C_{\mathbf{w}_{i}}(q^{\prime}q)_{\mathbf{w}_{i}}=\max_{\left|\mathbf{w}\right|=l_{i},\mathbf{g}\leq\mathbf{w}^{-1}}\#C_{\mathbf{w}}(q^{\prime}q)_{\mathbf{w}}$ that has length $l_i$ and satisfies $\mathbf{g} \leq \mathbf{w}_i^{-1}$. Since by the definition of $C_{\mathbf{w}_{i}}$ the equality $\#C_{\mathbf{v}}(q^{\prime}q)_{\mathbf{v}}=\#C_{\mathbf{w}_{i}}(q^{\prime}q)_{\mathbf{w}_{i}}$ holds for all $\mathbf{v}\in C_{\mathbf{w}_{i}}$, this element can be chosen in such a way that $\phi(P_{\mathbf{w}_{i}})\leq\phi(P_{\mathbf{v}})$ for all $\mathbf{v}\in C_{\mathbf{w}_{i}}$. Now, by picking a suitable subset $\mathcal{M}\subseteq W$ of elements $\mathbf{w}\in W$ with length $l_{i}$ and $\mathbf{g}\leq\mathbf{w}^{-1}$, the sum $\sum_{\mathbf{w}\in W\text{: }\left|\mathbf{w}\right|=l_{i},\mathbf{g}\leq\mathbf{w}^{-1}}(q^{\prime} q)_{\mathbf{w}}$ can be written as $\sum_{\mathbf{w}\in\mathcal{M}}\#C_{\mathbf{w}}(q^{\prime} q)_{\mathbf{w}}$. By the choice of $\mathbf{w}_{i}$ we hence have
\begin{eqnarray}
\nonumber
\sum_{\mathbf{w}\in W\text{: }\left|\mathbf{w}\right|=l_{i},\mathbf{g}\leq\mathbf{w}^{-1}}(q^{\prime}q)_{\mathbf{w}}\leq\left(l_{i}+1\right)^{\#S}\#C_{\mathbf{w}_{i}}(q^{\prime}q)_{\mathbf{w}_{i}}
\end{eqnarray}
which implies in combination with \eqref{*} that for all $i\in\mathbb{N}$,
\begin{eqnarray} \label{**}
\#C_{\mathbf{w}_{i}}(q^{\prime})^{l_{i}}q_{\mathbf{w}_{i}}\geq\frac{C^{l_{i}}}{\left(l_{i}+1\right)^{\#S}}\text{.}
\end{eqnarray}
It follows from Lemma \ref{operator} that the series $\sum_{\mathbf{w}\in W\text{: } \mathbf{g}\leq\mathbf{w}^{-1}}(q^{\prime})^{\left|\mathbf{w}\right|}\phi(P_{\mathbf{w}})$ converges. By the same argument as above we hence have that
\begin{eqnarray}
\nonumber
\limsup_{l}\left(\sum_{\mathbf{w}\in W\text{: }\left|\mathbf{w}\right|=l,\mathbf{g}\leq\mathbf{w}^{-1}}(q^{\prime})^{l}\phi(P_{\mathbf{w}})\right)^{1/l}<1\text{.}
\end{eqnarray}
One can therefore assume (by possibly going over to a further subsequence) that
\[\left(\sum_{\mathbf{w}\in W\text{: }\left|\mathbf{w}\right|=l_{i}\text{, }\mathbf{g}\leq\mathbf{w}^{-1}}(q^{\prime})^{l_{i}}\phi(P_{\mathbf{w}})\right)^{1/l_{i}}\leq L
\] for all $i\in\mathbb{N}$ where $0<L<1$ . But then, by the choice of $\mathbf{w}_{i}$,
\begin{eqnarray}
\nonumber
\#C_{\mathbf{w}_{i}}(q^{\prime})^{l_{i}}\phi(P_{\mathbf{w}_{i}})\leq(q^{\prime})^{l_{i}}\sum_{\mathbf{w}\in C_{\mathbf{w}_{i}}}\phi(P_{\mathbf{w}})\leq\sum_{\mathbf{w}\in W\text{: }\left|\mathbf{w}\right|=l_{i},\mathbf{g}\leq\mathbf{w}^{-1}}(q^{\prime})^{l_{i}}\phi(P_{\mathbf{w}})\leq L^{l_{i}}
\end{eqnarray}
and thus with \eqref{**}
\begin{eqnarray}
\nonumber
0\leq q_{\mathbf{w}_{i}}^{-1}\phi(P_{\mathbf{w}_{i}})<\frac{L^{l_{i}}}{\#C_{\mathbf{w}_{i}}(q^{\prime})^{l_{i}}q_{\mathbf{w}_{i}}}\leq\left(l_{i}+1\right)^{\#S}\left(\frac{L}{C}\right)^{l_{i}}\rightarrow0\text{.}
\end{eqnarray}
This implies the first part of the statement. The second part is an immediate consequence, since $\pi(\mathfrak{A}(W))$ is a quotient of $\mathfrak{A}(W)$. That finishes the proof.
\end{proof}

\begin{remark}
The proof of Proposition \ref{step 1} significantly simplifies in the case of single-parameters $q$. Indeed, if we follow the notation of Proposition \ref{step 1} and assume that $q_{s}=q_{t}$ for all $s,t\in S$, Lemma \ref{operator} implies that for $i\in\mathbb{N}$ and $0<q^{\prime}<1$,
\begin{eqnarray}
\nonumber
\sum_{\mathbf{w}\in W\text{: }\left|\mathbf{w}\right|=i,\mathbf{g}\leq\mathbf{w}^{-1}}(q^{\prime})^{i}\phi(P_{\mathbf{w}})\leq\phi(\mathbf{Q}_{q^{\prime}}^{\mathbf{g}})\text{.}
\end{eqnarray}
One can thus find an element $\mathbf{w}_{i}$ of length $i$ with $\mathbf{g}\leq\mathbf{w}_{i}^{-1}$ such that $\phi(P_{\mathbf{w}_{i}})\leq\left(\#L_{i}^{\mathbf{g}}(q^{\prime})^{i}\right)^{-1}\phi({\mathbf{Q}}_{q^{\prime}}^{\mathbf{g}})$ where $L_{i}^{\mathbf{g}}:=\left\{ \mathbf{w}\in W\mid\left|\mathbf{w}\right|=i\text{, }\mathbf{g}\leq\mathbf{w}_{i}^{-1}\right\}$. We get that
\begin{eqnarray}
\nonumber
q_{\mathbf{w}_{i}}^{-1}\phi(P_{\mathbf{w}_{i}})\leq\frac{\phi(\mathbf{Q}_{q^{\prime}}^{\mathbf{g}})}{\#L_{i}^{\mathbf{g}}(q^{\prime}q)_{\mathbf{w}_{i}}}\text{.}
\end{eqnarray}
The Cauchy-Hadamard formula (for radii of convergence of power series) implies that for increasing $i$, if $q^{\prime}$ is close enough to $1$, the expression on the right approaches $0$.
\end{remark}

\vspace{1mm}


\section{Central projections in Hecke C$^{\ast}$-algebras} \label{3}

In \cite{Raum} Raum and Skalski, generalizing the single-parameter results by Garncarek \cite{Gar}, proved that for a right-angled, irreducible, finite rank Coxeter system $(W,S)$ with at least three generators and $q=(q_{s})_{s\in S}\in\mathbb{R}_{>0}^{(W,S)}$ the corresponding Hecke-von Neumann algebra $\mathcal{N}_{q}(W)$ decomposes as $\mathcal{N}_{q}(W)\cong M\oplus\bigoplus_{\epsilon\in\left\{ -1,1\right\} ^{(W,S)}\text{: }\left|q_{\epsilon}\right|\in\mathcal{R}^{\prime}}\mathbb{C}$ where $q_{\epsilon}:=(\epsilon_{s}q_{s}^{\epsilon_{s}})_{s\in S}, \left|q_{\epsilon}\right|:=(q_{s}^{\epsilon_{s}})_{s\in S}$ and where $M$ is a factor. In particular, $\mathcal{N}_{q}(W)$ is a factor if and only if $q\in\mathbb{R}_{>0}^{(W,S)}\setminus\mathcal{R}^{\prime}$. It is a natural question whether for $q\in\mathcal{R}^{\prime}$ the central projections in $\mathcal{N}_{q}(W)$ are already contained in the corresponding Hecke C$^{\ast}$-algebra $C_{r,q}^{\ast}(W)$. We will prove this by using a Haagerup-type inequality from \cite{Mario}. We will further characterize the characters (i.e. unital, linear, multiplicative functionals) of right-angled Hecke C$^{\ast}$-algebras.


\subsection{The center of right-angled Hecke C$^{\ast}$-algebras}

\begin{theorem}[{\cite[Theorem 3.4]{Mario}}] \label{Haagerup}
Let $(W,S)$ be a right-angled, finite rank Coxeter system and let $q=(q_{s})_{s\in S}\in\mathbb{R}_{>0}^{(W,S)}$. Then there exists a constant $C>0$ such that for every $l\in\mathbb{N}_{\geq1}$ and $x\in C_{r,q}^{\ast}(W)$ of the form $x:=\sum_{\mathbf{w}\in W\text{: }\left|\mathbf{w}\right|=l}c_{\mathbf{w}}T_{\mathbf{w}}^{(q)}$ with coefficients $c_{\mathbf{w}}\in\mathbb{C}$ we have $\left\Vert x\right\Vert \leq Cl\left\Vert x \delta_e \right\Vert _{2}$.
\end{theorem}

We will further need the following easy lemma.

\begin{lemma} \label{openness}
Let $(W,S)$ be a finite rank Coxeter system. Then the intersection $\mathcal{R}\cap\mathbb{R}_{>0}^{(W,S)}$ of the region of convergence $\mathcal{R}$ of the growth series $W(z)=\sum_{\mathbf{w}\in W}z_{\mathbf{w}}$ with $\mathbb{R}_{>0}^{(W,S)}$ is open in $\mathbb{R}_{>0}^{(W,S)}$.
\end{lemma}

\begin{proof}
Assume that the set $\mathcal{R}\cap\mathbb{R}_{>0}^{(W,S)}$ is not open in $\mathbb{R}_{>0}^{(W,S)}$ and let $q \in\mathcal{R}\cap\mathbb{R}_{>0}^{(W,S)}$ be a point on its boundary. Since $\sum_{\mathbf{w}\in W}q_{\mathbf{w}}$ converges, the power series $f(z):=\sum_{\mathbf{w}\in W}q_{\mathbf{w}}z^{\left|\mathbf{w}\right|}$ absolutely converges for all $z\in\mathbb{C}$ with $\left|z\right|\leq1$. But the radius of convergence of $f$ coincides with the distance of the origin to the closest pole of $f$, hence there exists $\lambda>1$ such that $\sum_{\mathbf{w}\in W}q_{\mathbf{w}}z^{\left|\mathbf{w}\right|}$ absolutely converges for all $z\in\mathbb{C}$ with $\left|z\right|<\lambda$. This implies that $(2^{-1}(1+\lambda)q_{s})_{s\in S}\in\mathcal{R}\cap\mathbb{R}_{>0}^{(W,S)}$ which contradicts the choice of $q$.
\end{proof}

\begin{proposition} \label{projections}
Let $(W,S)$ be a right-angled, finite rank Coxeter system, let $q=(q_{s})_{s\in S}\in\mathbb{R}_{>0}^{(W,S)}$ with $0<q_{s}\leq1$ for all $s\in S$ be a multi-parameter and let $W(z)=\sum_{\mathbf{w}\in W}z_{\mathbf{w}}$ be the growth series of $(W,S)$. Further, let $\epsilon\in\left\{ -1,1\right\} ^{(W,S)}$, $q_{\epsilon}:=\left(\epsilon_{s}q_{s}^{\epsilon_{s}}\right)_{s\in S}$ and assume that $\left|q_{\epsilon}\right|:=\left(q_{s}^{\epsilon_{s}}\right)_{s\in S}\in\mathcal{R}$. Then the operator
\begin{eqnarray}
\nonumber
E_{q,\epsilon}=\frac{1}{W(\left|q_{\epsilon}\right|)}\sum_{i=0}^{\infty} \; \sum_{\mathbf{w}\text{: }\left|\mathbf{w}\right|=i}\left(\sqrt{q}\right)_{\mathbf{w},\epsilon}T_{\mathbf{w}}^{(q)}
\end{eqnarray}
exists (where the limit is taken with respect to the operator norm), it is a projection in $C_{r,q}^\ast (W)$ and it satisfies $T_{s}^{(q)}E_{q,\epsilon}=E_{q,\epsilon}T_{s}^{(q)}=\epsilon_{s}q_{s}^{\frac{\epsilon_{s}}{2}}E_{q,\epsilon}$ for all $s\in S$. For distinct $\epsilon,\epsilon^{\prime}\in\left\{ -1,1\right\} ^{(W,S)}$ with $\left|q_{\epsilon}\right|,\left|q_{\epsilon^{\prime}}\right|\in\mathcal{R}$ the projections $E_{q,\epsilon}$ and $E_{q,\epsilon^{\prime}}$ are orthogonal to each other.
\end{proposition}

\begin{proof}
By assumption $\left|q_{\epsilon}\right|\in\mathcal{R}$, so Lemma \ref{openness} implies that there exists $\lambda>1$ such that still $\left|\lambda q_{\epsilon}\right|:=\left(\lambda q_{s}^{\epsilon_{s}}\right)_{s\in S}\in\mathcal{R}$. Using the root test criterium for convergence,
\begin{eqnarray}
\nonumber
\limsup_{l}\left(\sum_{\mathbf{w}\in W\text{: }\left|\mathbf{w}\right|=l}\lambda^{l}\left|q_{\mathbf{w},\epsilon}\right|\right)^{1/l}\leq1
\end{eqnarray}
and hence
\begin{eqnarray}
\nonumber
\limsup_{l}\left(\sum_{\mathbf{w}\in W\text{: }\left|\mathbf{w}\right|=l}\left|q_{\mathbf{w},\epsilon}\right|\right)^{1/l}<1\text{.}
\end{eqnarray}
One can therefore find $l_{0}\in\mathbb{N}$ and $0<L<1$ such that for all $l\geq l_{0}$,
\begin{eqnarray} \label{***}
\sum_{\mathbf{w}\in W\text{: }\left|\mathbf{w}\right|=l}\left|q_{\mathbf{w},\epsilon}\right|<L^{l}\text{.}
\end{eqnarray}
Now set $E_{q,\epsilon}^{(i)}:=\left(W(\left|q_{\epsilon}\right|)\right)^{-1}\sum_{l=0}^{i}\sum_{\mathbf{w}\text{: }\left|\mathbf{w}\right|=l}\left(\sqrt{q}\right)_{\mathbf{w},\epsilon}T_{\mathbf{w}}^{(q)}$. For $i,j\in\mathbb{N}$ with $i<j$ and $i\geq l_{0}$ we have by Theorem \ref{Haagerup} and the inequality \eqref{***} that
\begin{eqnarray}
\nonumber
\left\Vert E_{q,\epsilon}^{(j)}-E_{q,\epsilon}^{(i)}\right\Vert &\leq& \frac{1}{W(\left|q_{\epsilon}\right|)}\sum_{l=i+1}^{j}\left\Vert \sum_{\mathbf{w}\text{: }\left|\mathbf{w}\right|=l}\left(\sqrt{q}\right)_{\mathbf{w},\epsilon}T_{\mathbf{w}}^{(q)}\right\Vert\\
\nonumber
&\leq& \frac{1}{W(\left|q_{\epsilon}\right|)}\sum_{l=i+1}^{j}Cl\sqrt{\sum_{\mathbf{w}\text{: }\left|\mathbf{w}\right|=l}\left|q_{\mathbf{w},\epsilon}\right|}\\
\nonumber
&<& \frac{1}{W(\left|q_{\epsilon}\right|)}\sum_{l=i+1}^{j}ClL^{\frac{l}{2}}\text{.}
\end{eqnarray}
The series $\sum_{l=0}^{\infty}lL^{\frac{l}{2}}$ converges, so $(E_{q,\epsilon}^{(i)})_{i \in\mathbb{N}}\subseteq C_{r,q}^{\ast}(W)$ converges to $E_{q,\epsilon}\in C_{r,q}^{\ast}(W)$. The remaining statements follow from short calculations (compare also with \cite[Lemma 19.2.5]{Davis}, \cite[Theorem 5.3]{Gar} and \cite[Proposition 2.2]{Raum}) that we omit here.
\end{proof}

\begin{remark}
In \cite{Raum} Raum and Skalski introduced the notion of \emph{Hecke eigenvectors} (for the parameter $q$). These are non-zero elements $\eta \in \ell^2 (W)$ satisfying $T_s^{(q)}\eta \in \mathbb{C} \eta$ for all $s \in S$. The central projections considered in Proposition \ref{projections} are exactly the orthogonal projections onto the Hecke eigenspaces. Note that they are always of finite rank.
\end{remark}

The following corollary follows from \cite[Theorem A]{Raum} and Proposition \ref{isomorphism}.

\begin{corollary}
Let $(W,S)$ be a right-angled, finite rank Coxeter system with $\#S\geq3$ and let $q\in\mathbb{R}_{>0}^{(W,S)}$. Then the center of the Hecke C$^{\ast}$-algebra $C_{r,q}^{\ast}(W)$ coincides with the center of the Hecke-von Neumann algebra $\mathcal{N}_{q}(W)$.
\end{corollary}

One other immediate consequence is that right-angled Hecke C$^{\ast}$-algebras admit a decomposition which is analogous to the one of their von Neumann-algebraic counterparts.

\begin{corollary}
Let $(W,S)$ be a right-angled, finite rank Coxeter system with $\#S\geq3$ and let $q=(q_{s})_{s\in S}\in\mathbb{R}_{>0}^{(W,S)}$. Then the corresponding Hecke C$^{\ast}$-algebra $C_{r,q}^{\ast}(W)$ decomposes as
\begin{eqnarray}
\nonumber
C_{r,q}^{\ast}(W)\cong\pi(C_{r,q}^{\ast}(W))\oplus\bigoplus_{\epsilon\in\left\{ -1,1\right\}^{(W,S)} \text{: }\left|q_{\epsilon}\right|\in\mathcal{R}^{\prime}}\mathbb{C},
\end{eqnarray}
where $\pi$ denotes the quotient map of $\mathcal{B}(\ell^{2}(W))$ onto $\mathcal{B}(\ell^{2}(W))/\mathcal{K}(\ell^{2}(W))$.
\end{corollary}

\begin{proof}
By Proposition \ref{projections} $C_{r,q}^{\ast}(W)$ decomposes as $A\oplus\bigoplus_{\epsilon\in\left\{ -1,1\right\}^{(W,S)} \text{: }\left|q_{\epsilon}\right|\in\mathcal{R}^{\prime}}\mathbb{C}$ where
\begin{eqnarray}
\nonumber
A=C_{r,q}^{\ast}(W)\prod_{\epsilon\in\left\{ -1,1\right\}^{(W,S)} \text{: }\left|q_{\epsilon}\right|\in\mathcal{R}^{\prime}}(1-E_{q,\epsilon})\subseteq\mathcal{B}(\ell^{2}(W)).
\end{eqnarray}
By \cite[Theorem A]{Raum} the von Neumann algebra $A^{\prime\prime}\subseteq\mathcal{B}(\ell^{2}(W))$ is a factor, necessarily of type $\text{II}_1$, so $A$ contains no compact operators. This implies that $A\cong\pi(C_{r,q}^{\ast}(W))$ from which the claim follows.
\end{proof}


\subsection{Characters on Hecke C$^{\ast}$-algebras} \label{characters}

The operators appearing in Proposition \ref{projections} are projections onto one-dimensional subspaces of $\ell^{2}(W)$ and thus induce characters on the right-angled Hecke C$^{\ast}$-algebras. Let us prove that all characters on right-angled Hecke C$^{\ast}$-algebras arise in such a manner.

\begin{proposition} \label{characters}
Let $(W,S)$ be a right-angled, irreducible, finite rank Coxeter system and $q=(q_{s})_{s\in S}\in\mathbb{R}_{>0}^{(W,S)}$. Then the set of characters of the corresponding Hecke C$^{\ast}$-algebra $C_{r,q}^{\ast}(W)$ is given by
\begin{eqnarray}
\nonumber
\{ \chi_{q_{\epsilon}}\mid\epsilon\in\left\{ -1,1\right\} ^{(W,S)}\text{ with }\left|q_{\epsilon}\right|\in\overline{\mathcal{R}^{\prime}}\} \text{,}
\end{eqnarray}
where $\left|q_{\epsilon}\right|:=(q_{s}^{\epsilon_{s}})_{s\in S}\in\mathbb{R}_{>0}^{(W,S)}$ and $\chi_{q_{\epsilon}}$ satisfies $\chi_{q_{\epsilon}}(T_{s}^{(q)}):=\epsilon_{s}q_{s}^{\frac{\epsilon_{s}}{2}}$ for all $s\in S$.
\end{proposition}

\begin{proof}
By Proposition \ref{isomorphism} we can assume that $0<q_{s}\leq1$ for all $s\in S$. Arguing exactly as in the proof of \cite[Lemma 5.3]{Mario} it follows from Proposition \ref{projections} (or also from \cite[Proposition 2.2]{Raum}) that for every $\epsilon\in\left\{ -1,1\right\} ^{(W,S)}$ with $\left|q_{\epsilon}\right|\in\overline{\mathcal{R}^{\prime}}$ the character $\chi_{q_{\epsilon}}$ exists. Conversely, let $\chi$ be a state on $\mathfrak{A}(W)$ which restricts to a character on $C_{r,q}^{\ast}(W)$. For $s\in S$ the Hecke relation $(T_{s}^{(q)})^{2}=1+p_{s}(q)T_{s}^{(q)}$ implies $(\chi(T_{s}^{(q)}))^{2}-p_{s}(q)\chi(T_{s}^{(q)})=1$ and hence $\chi(T_{s}^{(q)})\in\{ q_{s}^{\frac{1}{2}},-q_{s}^{-\frac{1}{2}}\}$. One can thus find $\epsilon\in\left\{ -1,1\right\} ^{(W,S)}$ with $\chi=\chi_{q_{\epsilon}}$. Now assume that $\left|q_{\epsilon}\right|\notin\overline{\mathcal{R}^{\prime}}$, fix $s\in S$ and choose a path $\mathbf{g}:=s_{1}...s_{n}\in W$ in the Coxeter diagram of $(W,S)$ that covers the whole graph and for which $m_{ss_{1}}=\infty$. By Proposition \ref{step 1} there exists a sequence $(\mathbf{w}_{i})_{i\in\mathbb{N}}\subseteq W$ of increasing word length with $\mathbf{g}\leq\mathbf{w}_{i}^{-1}$ for all $i\in\mathbb{N}$ and $\left|q_{\mathbf{w}_{i},\epsilon}^{-1}\chi(P_{\mathbf{w}_{i}})\right|\rightarrow0$. We further have that $P_{\mathbf{w}_{i}s}\leq P_{\mathbf{w}_{i}}$, so $\chi(P_{\mathbf{w}_{i}s})\leq\chi(P_{\mathbf{w}_{i}})$. Using that $T_{\mathbf{w}_{i}}^{(q)}$ and $T_{\mathbf{w}_{i}^{-1}}^{(q)}$ lie in the multiplicative domain of $\chi$ (see for instance \cite[Proposition 1.5.7]{BrownOzawa}) one has
\begin{eqnarray}
\nonumber
\left|\chi(T_{s}^{(q)})\right|&=&\left|q_{\mathbf{w}_{i},\epsilon}^{-1}\right|\left|\chi(T_{\mathbf{w}_{i}}^{(q)}(1-P_{s})T_{s}^{(q)}T_{\mathbf{w}_{i}^{-1}}^{(q)})+\chi(T_{\mathbf{w}_{i}}^{(q)}P_{s}T_{s}^{(q)}T_{\mathbf{w}_{i}^{-1}}^{(q)})\right| \\
\nonumber
&\leq&  \left|q_{\mathbf{w}_{i},\epsilon}^{-1}\chi(T_{\mathbf{w}_{i}}^{(q)}(1-P_{s})T_{s}^{(q)}T_{\mathbf{w}_{i}^{-1}}^{(q)})\right|+\left|q_{\mathbf{w}_{i},\epsilon}^{-1}\chi(T_{\mathbf{w}_{i}}^{(q)}P_{s}T_{s}^{(q)}T_{\mathbf{w}_{i}^{-1}}^{(q)})\right|.
\end{eqnarray}
By Proposition \ref{action} (as well as Remark \ref{remark2}),
\begin{eqnarray}
\nonumber
(1-P_{s})T_{s}^{(q)}T_{\mathbf{w}_{i}^{-1}}^{(q)}=T_{s\mathbf{w}_{i}^{-1}}^{(1)}P_{\mathbf{w}_{i}s} \:\:\: \text{and} \:\:\: T_{\mathbf{w}_{i}}^{(q)}P_{s}=P_{\mathbf{w}_{i}s}T_{\mathbf{w}_{i}}^{(1)},
\end{eqnarray}
so
\begin{eqnarray}
\nonumber
\left|\chi(T_{s}^{(q)})\right| &\leq& \left|q_{\mathbf{w}_{i},\epsilon}^{-1}\chi(T_{\mathbf{w}_{i}}^{(q)} T_{s \mathbf{w}_{i}^{-1}}^{(1)}P_{\mathbf{w}_{i}s})\right|+\left|q_{\mathbf{w}_{i},\epsilon}^{-1}\chi(P_{\mathbf{w}_{i}s}T_{\mathbf{w}_{i}}^{(1)}T_{s}^{(q)}T_{\mathbf{w}_{i}^{-1}}^{(q)})\right| \\
\nonumber
&=& \left|q_{\mathbf{w}_{i},\epsilon}^{-1/2}\chi(T_{s\mathbf{w}_{i}^{-1}}^{(1)}P_{\mathbf{w}_{i}s})\right|+\left|q_{s, \epsilon}^{-1/2} q_{\mathbf{w}_{i},\epsilon}^{-1/2}\chi(P_{\mathbf{w}_{i}s}T_{\mathbf{w}_{i}}^{(1)})\right| \text{.}
\end{eqnarray}
The Cauchy-Schwarz inequality then implies
\begin{eqnarray}
\nonumber
\left|\chi(T_{s}^{(q)})\right|&\leq& (1+q_{s}^{-1/2})\sqrt{\left|q_{\mathbf{w}_{i},\epsilon}^{-1}\chi\left(P_{\mathbf{w}_{i}s}\right)\right|} \rightarrow 0 \text{.}
\end{eqnarray}
This contradicts $\chi(T_{s}^{(q)})\in\{ q_{s}^{\frac{1}{2}},-q_{s}^{-\frac{1}{2}}\}$.
\end{proof}

\vspace{1mm}


\section{Simplicity of right-angled Hecke C$^{\ast}$-algebras} \label{simplicity}

In this last section we study the simplicity of right-angled Hecke C$^{\ast}$-algebras (recall that a C$^{\ast}$-algebra is \emph{simple} if it contains no non-trivial, two-sided, closed ideal). Our approach is inspired by \cite{Haagerup}. It requires the following lemma which immediately follows from \cite[Lemma 4.2]{Mario2}.

\begin{lemma} \label{positive}
Let $(W,S)$ be a Coxeter system, $q=(q_{s})_{s\in S}\in\mathbb{R}_{>0}^{(W,S)}$ and $\mathbf{w}\in W$. Let further $\mathbf{w}=s_{1}...s_{n}$ with $s_{1},...,s_{n}\in S$ be a reduced expression for $\mathbf{w}$. Then,
\begin{eqnarray}
\nonumber
\prod_{i=1}^{n}\min\left\{ q_{s_{i}}^{\pm1}\right\} \leq(T_{\mathbf{w}}^{(q)})^{\ast}T_{\mathbf{w}}^{(q)}\leq\prod_{i=1}^{n}\max\left\{ q_{s_{i}}^{\pm 1}\right\} \text{.}
\end{eqnarray}
\end{lemma}

Recall that for a finite rank Coxeter system $(W,S)$ and $q\in\mathbb{R}_{>0}^{(W,S)}\setminus\mathcal{R}^{\prime}$ the Hecke C$^{\ast}$-algebra $C_{r,q}^{\ast}(W)$ can be viewed as a C$^{\ast}$-subalgebra of $\pi(\mathfrak{A}(W))$ (see Section \ref{2}). We will use this observation frequently.

\begin{proposition} \label{step 2}
Let $(W,S)$ be a right-angled, irreducible, finite rank Coxeter system, $q=(q_{s})_{s\in S}\in\mathbb{R}_{>0}^{(W,S)}\setminus\overline{\mathcal{R}^{\prime}}$ with $0<q_{s}\leq1$ for all $s\in S$ and let $I\neq C_{r,q}^{\ast}(W)$ be an ideal of $C_{r,q}^\ast (W)$ where we view $C_{r,q}^{\ast}(W)$ as a C$^{\ast}$-subalgebra of $\pi(\mathfrak{A}(W))$. Then, for every two elements $s,t\in S$ with $m_{st}=\infty$ there exists a state $\phi$ on $\pi(\mathfrak{A}(W))$ that vanishes on $I$ and which satisfies $\phi(P_{s})=1$, $\phi(P_{t})=0$.
\end{proposition}

\begin{proof}
Choose a state on $C_{r,q}^{\ast}(W)$ that vanishes on $I$. We can extend it to a state $\psi$ on $\pi(\mathfrak{A}(W))$. Further let $\mathbf{g}:=s_{1}...s_{n}\in W$ with $s_{1}:=s$, $s_{2}:=t$ be a path in the Coxeter diagram of $(W,S)$ that covers the whole graph and let $(\mathbf{w}_{i})_{i\in\mathbb{N}}\subseteq W$ be a sequence as in Proposition \ref{step 1}, i.e. the $\mathbf{w}_{i}$ have increasing word length, $\mathbf{g}\leq\mathbf{w}_{i}^{-1}$ for all $i\in\mathbb{N}$ and $q_{\mathbf{w}_{i}}^{-1}\psi(\tilde{P}_{\mathbf{w}_{i}})\rightarrow0$. Note that $\psi(T_{\mathbf{w}_{i}}^{(q)}T_{\mathbf{w}_{i}^{-1}}^{(q)})=0$ is not possible since then Lemma \ref{positive} and the Cauchy-Schwarz inequality would imply
\begin{eqnarray}
\nonumber
0=\psi(T_{\mathbf{w}_{i}}^{(q)}T_{\mathbf{w}_{i}^{-1}}^{(q)})\geq q_{\mathbf{w}_{i}s_{1}}\psi((T_{s_{1}}^{(q)})^{2})\geq q_{\mathbf{w}_{i}s_{1}}|\psi(T_{s_{1}}^{(q)})|^{2}
\end{eqnarray}
and thus $\psi((T_{s_{1}}^{(q)})^{2})=\psi(T_{s_{1}}^{(q)})=0$. This contradicts the identity $(T_{s_{1}}^{(q)})^{2}=1+p_{s}\left(q\right)T_{s_{1}}^{(q)}$.  With Proposition \ref{action} (as well as Remark \ref{remark2}) and Lemma \ref{positive} we get that for $i\in\mathbb{N}$,
\begin{eqnarray}
\nonumber
\left|\frac{\psi(T_{\mathbf{w}_{i}}^{(q)}\tilde{P}_{s}T_{\mathbf{w}_{i}^{-1}}^{(q)})}{\psi(T_{\mathbf{w}_{i}}^{(q)}T_{\mathbf{w}_{i}^{-1}}^{(q)})}-1\right|=\left|\frac{\psi(T_{\mathbf{w}_{i}}^{(q)}(\tilde{P}_{s}-1)T_{\mathbf{w}_{i}^{-1}}^{(q)})}{\psi(T_{\mathbf{w}_{i}}^{(q)}T_{\mathbf{w}_{i}^{-1}}^{(q)})}\right|=\left|\frac{\psi(\tilde{P}_{\mathbf{w}_{i}})}{\psi(T_{\mathbf{w}_{i}}^{(q)}T_{\mathbf{w}_{i}^{-1}}^{(q)})}\right|\leq q_{\mathbf{w}_{i}}^{-1}\psi(\tilde{P}_{\mathbf{w}_{i}})\rightarrow0\text{.}
\end{eqnarray}
The weak-$\ast$ compactness of the state space $\mathcal{S}(\pi(\mathfrak{A}(W)))$ implies that we can find a subsequence of
\begin{eqnarray}
\nonumber
\left((\psi(T_{\mathbf{w}_{i}}^{(q)}T_{\mathbf{w}_{i}^{-1}}^{(q)}))^{-1}\psi(T_{\mathbf{w}_{i}}^{(q)}(\cdot)T_{\mathbf{w}_{i}^{-1}}^{(q)})\right)_{i\in\mathbb{N}}\subseteq\mathcal{S}(\pi(\mathfrak{A}(W)))
\end{eqnarray}
that weak-$\ast$ converges to a state $\phi$. By construction, this state vanishes on the ideal $I$, we have $\phi(\tilde{P}_{s})=1$ and hence also $\phi(\tilde{P}_{t})=0$ since $0\leq\tilde{P}_{t}\leq1-\tilde{P}_{s}$.
\end{proof}

Recall that the inner action of the group $W$ on $\pi(\mathfrak{A}(W))$ defined by $\mathbf{w}.x:=T_{\mathbf{w}}^{(1)}xT_{\mathbf{w}^{-1}}^{(1)}$ for $\mathbf{w}\in W$, $x\in\pi(\mathfrak{A}(W))$ induces an action of $W$ on the state space of $\pi(\mathfrak{A}(W))$ via $(\mathbf{w}.\phi)(x):=\phi(T_{\mathbf{w}^{-1}}^{(1)}xT_{\mathbf{w}}^{(1)})$ for $\phi\in\mathcal{S}(\pi(\mathfrak{A}(W)))$, $\mathbf{w}\in W$ and $x\in\pi(\mathfrak{A}(W))$.

We are now ready to characterize the simplicity of right-angled Hecke C$^{\ast}$-algebras.

\begin{theorem} \label{simplicity}
Let $(W,S)$ be an irreducible, right-angled, finite rank Coxeter system and let $q=(q_{s})_{s\in S}\in\mathbb{R}_{>0}^{(W,S)}$ be a multi-parameter. Then the Hecke C$^{\ast}$-algebra $C_{r,q}^{\ast}(W)$ is simple if and only if $q\in\mathbb{R}_{>0}^{(W,S)}\setminus\overline{\mathcal{R}^{\prime}}$.
\end{theorem}

\begin{proof}
By Proposition \ref{characters} the Hecke C$^{\ast}$-algebra $C_{r,q}^{\ast}(W)$ is not simple for $q\in\overline{\mathcal{R}^{\prime}}$. For the treatment of the case where $q\in\mathbb{R}_{>0}^{(W,S)}\setminus\overline{\mathcal{R}^{\prime}}$ by Proposition \ref{isomorphism} it suffices to consider multi-parameters with $0<q_{s}\leq1$ for all $s\in S$. View $C_{r,q}^{\ast}(W)$ as a C$^{\ast}$-subalgebra of $\pi(\mathfrak{A}(W))$ and assume that $I\neq C_{r,q}^{\ast}(W)$ is an ideal in $C_{r,q}^{\ast}(W)$. Further choose a closed path $\mathbf{g}:=s_{1}...s_{n}$ in the Coxeter diagram of $(W,S)$ that covers the whole graph. Proposition \ref{step 2} implies that we can find a state $\phi$ on $\pi(\mathfrak{A}(W))$ that vanishes on $I$ and for which $\phi(\tilde{P}_{s_{1}})=1$, $\phi(\tilde{P}_{s_{n}})=0$ holds. In particular the projections $\tilde{P}_{s_{1}}$, $\tilde{P}_{s_{n}}$ are contained in the multiplicative domain of $\phi$ (see for instance \cite[Proposition 1.5.7]{BrownOzawa}).

By the identification $\pi(\mathcal{D}(W,S))\cong C(\partial(W,S))$ and the equality $\phi(\tilde{P}_{s_{1}})=1$ the restriction of $\phi$ to $\pi(\mathcal{D}(W,S))$ corresponds to a probability measure $\mu$ on the boundary $\partial(W,S)$ whose support is contained in the set of all $z\in\partial(W,S)$ with $s_{1}\leq z$. The sequence $(\mathbf{g}^{i}.\mu)_{i\in\mathbb{N}}$ hence weak-$\ast$ converges to the point mass $\delta_{\mathbf{g}^{\infty}}\in\text{Prob}(\partial(W,S))$ where $g^{\infty}:=\lim_{l}\mathbf{g}^{l}\in\partial(W,S)$ and where $\text{Prob}(\partial(W,S))$ denotes the space of all probability measures on $\partial(W,S)$ (compare also with the proof of \cite[Theorem 3.20]{Mario2}). This implies that there exists an increasing sequence $(i_{k})_{k\in\mathbb{N}}\subseteq\mathbb{N}$ for which $\left(\mathbf{g}^{i_{k}}.\phi\right)_{k\in\mathbb{N}}$ weak-$\ast$ converges to a state $\psi$ whose restriction to $\pi(\mathcal{D}(W,S))$ is multiplicative. The product $s_n ... s_1 $ also defines a path in the Coxeter diagram of $(W,S)$. Using Lemma \ref{decomposition} and $\phi(\tilde{P}_{s_{n}})=0$ one deduces that for $a \in I$
\begin{eqnarray}
\nonumber
\psi(a)=\lim_{k}\phi(T_{\mathbf{g}^{-i_{k}}}^{(1)}aT_{\mathbf{g}^{i_{k}}}^{(1)})=\lim_{k}\phi(T_{\mathbf{g}^{-i_{k}}}^{(q)}aT_{\mathbf{g}^{i_{k}}}^{(q)})=0\text{,}
\end{eqnarray}
so $\psi$ vanishes on the ideal $I$.

Now, let $J$ be the ideal in $\pi(\mathfrak{A}(W))$ generated by $I$. Since $\pi(\mathfrak{A}(W))$ identifies with the crossed product C$^{\ast}$-algebra $C(\partial(W,S)) \rtimes_{r} W$, every element in $\pi(\mathfrak{A}(W))$ can be approximated by a finite sum of the form $\sum_{\mathbf{w}\in W}f_{\mathbf{w}}T_{\mathbf{w}}^{(1)}$ where $f_{\mathbf{w}}\in \pi(\mathcal{D}(W,S))$. Using $T_{s}^{(1)}=T_{s}^{(q)}-p_{s} (q) P_{s}$ for $s\in S$, one concludes via induction that every such operator can be written as a finite sum the form $\sum_{\mathbf{w}\in W}g_{\mathbf{w}}T_{\mathbf{w}}^{(q)}$ for suitable $g_{\mathbf{w}}\in \pi(\mathcal{D}(W,S))$. But for all $a\in I$, $g,h\in \pi(\mathcal{D}(W,S))$ and $\mathbf{v},\mathbf{w}\in W$ we have that
\begin{eqnarray}
\nonumber
\psi((gT_{\mathbf{w}}^{(q)})a(T_{\mathbf{v}}^{(q)}h))=\psi(g)\psi(T_{\mathbf{w}}^{(q)}aT_{\mathbf{v}}^{(q)})\psi(h)=0
\end{eqnarray}
since $T_{\mathbf{w}}^{(q)}aT_{\mathbf{v}}^{(q)}\in I$, so the state $\psi$ vanishes on $J$. In particular, since $\psi\neq0$, $J$ can not coincide with the whole C$^{\ast}$-algebra $\pi(\mathfrak{A}(W))$. But $\pi(\mathfrak{A}(W))$ is simple by \cite[Corollary 4.11]{Mario2}, so $J=0$. We get that $C_{r,q}^{\ast}(W)$ must be simple as well. This completes the proof.
\end{proof}

\begin{corollary} \label{infinitely}
Let $(W,S)$ be an irreducible, right-angled Coxeter
system with $\#S=\infty$ and let $q=(q_{s})_{s\in S}\in\mathbb{R}_{>0}^{(W,S)}$.
Then the Hecke C$^{\ast}$-algebra $C_{r,q}^{\ast}(W)$ is simple
if and only if there exists a finite subset $T\subseteq S$ such that
the Hecke C$^{\ast}$-algebra $C_{r,q_{T}}^{\ast}(W_{T})$ with $q_{T}:=(q_{t})_{t\in T}$
is simple.
\end{corollary}

\begin{proof}
Again, by Proposition \ref{isomorphism} it suffices to consider
multi-parameters with $0<q_{s}\leq1$ for $s\in S$. First assume
that for all finite subsets $T\subseteq S$ the Hecke C$^{\ast}$-algebra
$C_{r,q_{T}}^{\ast}(W_{T})$ is not simple. The map $\chi:T_{\mathbf{w}}^{(q)}\mapsto q_{\mathbf{w}}^{\frac{1}{2}}$,
$\mathbf{w}\in W$ defines a character on $\mathbb{C}_{q}[W]$. Further,
for every element $x:=\sum_{\mathbf{w}\in W}x(\mathbf{w})T_{\mathbf{w}}^{(q)}\in\mathbb{C}_{q}[W]$ with $x(\mathbf{w}) \in \mathbb{C}$ for all $\mathbf{w} \in W$
there exists a finite subset $T\subseteq S$ such that the support
$\{\mathbf{w}\in W\mid x(\mathbf{w})\neq0\}$ of $x$ is contained
in the special subgroup $W_{T}$. Recall that $C_{r,q_{T}}^{\ast}(W_{T})$ canonically
embeds into $C_{r,q}^{\ast}(W)$. Since by the assumption
$C_{r,q_{T}}^{\ast}(W_{T})$ is not simple, Theorem \ref{simplicity} implies in
combination with Proposition \ref{characters} that the restriction of $\chi$ to
$\mathbb{C}_{q_{T}}[W_{T}]$ continuously
extends to a character $\chi_{T}$ on $C_{r,q_{T}}^{\ast}(W_{T})$.
But then, $\left|\chi(x)\right|=\left|\chi_{T}(x)\right|\leq\left\Vert x\right\Vert $,
so (as $x$ was arbitrary) $\chi$ continuously extends to a character
on $C_{r,q}^{\ast}(W)$. Hence $C_{r,q}^\ast (W)$ is not simple.

Conversely assume that there exists a finite subset $T\subseteq S$
for which $C_{r,q_{T}}^{\ast}(W_{T})$ is simple. Then from Theorem
\ref{simplicity} it follows that the C$^{\ast}$-algebra $C_{r,q_{T^{\prime}}}^{\ast}(W_{T^{\prime}})$
is simple for all finite subsets $T^{\prime}\subseteq S$ with $T\subseteq T^{\prime}$.
It is a standard fact that inductive limits of simple C$^{\ast}$-algebras
are simple (see e.g. \cite[II.8.2.5]{Blackadar}), so the simplicity of $C_{r,q}^{\ast}(W)$
follows from Lemma \ref{isomorphism3}.
\end{proof}

The following example demonstrates that there exist infinitely
generated right-angled, irreducible Coxeter systems and corresponding
multi-parameters whose respective Hecke C$^{\ast}$-algebras are non-simple.

\begin{example}
Let $S=\{s_{1},s_{2},...\}$ be a countable set
and consider the Coxeter group $W$ generated by $S$ subject to the
relations defined by $m_{ss}=2$ for all $s\in S$ and $m_{st}=\infty$
for all $s,t\in S$, $s\neq t$. Define $q:=(q_{s})_{s\in S}\in\mathbb{R}_{>0}^{(W,S)}$
by $q_{s_{i}}:=2^{-i}$ for $i\in\mathbb{N}_{\geq1}$. Then for every
finite subset $T\subseteq S$ one checks that
\[
\sum_{s\in T}\frac{1}{1+q_{s}}\geq\sum_{i=1}^{\#T}\frac{1}{1+2^{-i}}\geq\#T-1
\]
and hence, by the analysis in \cite[Subsection 5.4]{Mario}, the C$^{\ast}$-algebra $C_{r,q_{T}}^{\ast}(W_{T})$
is not simple. Corollary \ref{infinitely} then implies that $C_{r,q}^{\ast}(W)$
is not simple.
\end{example}

\vspace{1mm}


\subsection*{Acknowledgements}

I am grateful to my supervisor Martijn Caspers for many enlightening discussions and helpful feedback.
It is also a pleasure to thank Sven Raum and Adam Skalski for their comments on an earlier draft of this paper.

\vspace{1mm}




\begin{thebibliography}{00}

\bibitem{BK} A. Bearden, M. Kalantar, \emph{Topological boundaries of unitary representations}, Int. Math. Res. Not. IMRN 2021, no. 12, 9425--9457.

\bibitem{Bernstein} J. N. Bernstein, \emph{Le ``centre'' de Bernstein}, Representations of reductive groups over a local field, 1–32, Travaux en Cours, Hermann, Paris, 1984.

\bibitem{combinatorics} A. Bj\"{o}rner, F. Brenti, \emph{Combinatorics of Coxeter groups}, Graduate Texts in Mathematics, 231. Springer, New York, 2005.

\bibitem{Blackadar}  B. Blackadar, \emph{Operator algebras. Theory of C$^\ast$-algebras and von Neumann algebras}, Encyclopaedia of Mathematical Sciences, 122. Operator Algebras and Non-commutative Geometry, III. Springer-Verlag, Berlin, 2006. xx+517 pp.

\bibitem{Bourbaki} N. Bourbaki, \emph{Lie groups and Lie algebras. Chapters 4-6}, Translated from the 1968 French original by Andrew Pressley, Elements of Mathematics (Berlin), Springer-Verlag, Berlin, 2002. xii+300 pp.

\bibitem{BKKO} E. Breuillard, M. Kalantar, M. Kennedy, N. Ozawa, \emph{C$^\ast$-simplicity and the unique trace property for discrete groups}, Publ. Math. Inst. Hautes \'{E}tudes Sci. 126 (2017), 35–71.

\bibitem{BrownOzawa} N. Brown, N. Ozawa, \emph{C$^\ast$-algebras and finite-dimensional approximations}, Graduate Studies in Mathematics, 88. American Mathematical Society, Providence, RI, 2008. xvi+509 pp.

\bibitem{Caprace} P. E. Caprace, J. L\'ecureux, \emph{Combinatorial and group-theoretic compactifications of buildings}, Ann. Inst. Fourier (Grenoble) 61 (2011), no. 2, 619--672.

\bibitem{Caspers} M. Caspers, \emph{Absence of Cartan subalgebras for right-angled Hecke von Neumann algebras}, Anal. PDE 13 (2020), no. 1, 1-28.

\bibitem{Mario} M. Caspers, M. Klisse, N.S. Larsen, \emph{Graph product Khintchine inequalities and Hecke C$^\ast$-algebras: Haagerup inequalities, (non)simplicity, nuclearity and exactness}, J. Funct. Anal. 280 (2021), no. 1, 108795, 41 pp.

\bibitem{CSW} M. Caspers, A. Skalski, M. Wasilewski, \emph{On masas in $q$-deformed von Neumann algebras}, Pacific J. Math. 302 (2019), no. 1, 1--21.

\bibitem{Cornulier} Y. De Cornulier, \emph{Semisimple Zariski closure of Coxeter groups}, arXiv preprint arXiv:1211.5635 (2012).

\bibitem{Davis} M. W. Davis, \emph{The Geometry and Topology of Coxeter groups}, London Mathematical Society Monographs Series, 32. Princeton University Press, Princeton, NJ, 2008.

\bibitem{Dymara2} M. W. Davis, J. Dymara, T. Januszkiewicz, B. Okun,\emph{Weighted $L^2$-cohomology of Coxeter groups}, Geom. Topol. 11 (2007), 47--138.

\bibitem{Dykema} K. Dykema, \emph{Free products of hyperfinite von Neumann algebras and free dimension}, Duke Math. J. 69 (1993), no. 1, 97--119.

\bibitem{Dykema2} K. Dykema, \emph{Simplicity and the stable rank of some free product C$^\ast$-algebras}, Trans. Amer. Math. Soc. 351 (1999), no. 1, 1--40.

\bibitem{Dymara} J. Dymara, \emph{Thin buildings}, Geom. Topol. 10 (2006), 667--694. 

\bibitem{Fe} G. Fendler, \emph{Simplicity of the reduced C$^\ast$-algebras of certain Coxeter groups}, Illinois J. Math. 47 (2003), no. 3, 883--897.

\bibitem{Gar} L. Garncarek, \emph{Factoriality of Hecke-von Neumann algebras of right-angled Coxeter groups}, J. Funct. Anal. 270 (2016), no. 3, 1202--1219.

\bibitem{Haagerup} U. Haagerup, \emph{A new look at C$^\ast$-simplicity and the unique trace property of a group}, Operator algebras and applications - the Abel Symposium 2015, 167-176, Abel Symp., 12, Springer, 2017.

\bibitem{DeLaHarpe} P. de La Harpe, \emph{On simplicity of reduced C$^\ast$‐algebras of groups}, Bull. London Math. Soc. 39 (2007), no. 1, 1--26.

\bibitem{HaKa} Y. Hartman, M. Kalantar, \emph{Stationary C$^\ast$-dynamical systems}, to appear in the Journal of the European Mathematical Society (JEMS).

\bibitem{IwahoriMatsumoto} N. Iwahori, H. Matsumoto, \emph{On some Bruhat decomposition and the structure of Hecke rings of $\mathfrak{p}$-adic Chevalley groups}, Inst. Hautes \'{E}tudes Sci. Publ. Math. No. 25 (1965), 5--48.

\bibitem{Adam} M. Kalantar, P. Kasprzak, A. Skalski, R. Vergnioux, \emph{Noncommutative Furstenberg boundary}, to appear in Analysis and PDE.

\bibitem{KalantarKennedy} M. Kalantar, M. Kennedy, \emph{Boundaries of reduced C$^\ast$-algebras of discrete groups}, J. Reine Angew. Math. 727 (2017), 247--267.

\bibitem{KL2} D. Kazhdan, G. Lusztig, \emph{Representations of Coxeter groups and Hecke algebras}, Invent. Math. 53 (1979), no. 2, 165--184.

\bibitem{KL} D. Kazhdan, G. Lusztig, \emph{Proof of the Deligne-Langlands conjecture for Hecke algebras}, Invent. Math. 87 (1987), no. 1, 153--215.

\bibitem{Kennedy} M. Kennedy, \emph{An intrinsic characterization of C$^\ast$-simplicity}, Ann. Sci. \'{E}c. Norm. Sup\'{e}r. (4) 53 (2020), no. 5, 1105-1119.

\bibitem{Mario2} M. Klisse, \emph{Topological boundaries of connected graphs and Coxeter groups}, to appear in the Journal of Operator Theory.

\bibitem{Lam} T. Lam, A. Thomas, \emph{Infinite reduced words and the Tits boundary of a Coxeter group}, Int. Math. Res. Not. IMRN 2015, no. 17, 7690--7733.

\bibitem{Matsumoto} H. Matsumoto, \emph{Analyse harmonique dans les syst\`{e}mes de Tits bornologiques de type affine}, Lecture Notes in Mathematics, Vol. 590, Springer-Verlag, Berlin-New York, 1977. i+219 pp.

\bibitem{Radulescu} F. Radulescu, \emph{Random matrices, amalgamated free products and subfactors of the von Neumann algebra of a free group, of noninteger index}, Invent. Math. 115 (1994), no. 2, 347--389.

\bibitem{Raum} S. Raum, A. Skalski, \emph{Factorial multiparameter Hecke von Neumann algebras and representations of groups acting on right-angled buildings}, arXiv preprint arXiv: 2008.00919 (2020).

\bibitem{Raum2} S. Raum, A. Skalski, \emph{K-theory of right-angled Hecke C$^\ast$-algebras}, arXiv preprint arXiv:2103.02962 (2021).

\bibitem{Remy} B. R\'{e}my, \emph{Buildings and Kac-Moody groups}, Buildings, finite geometries and groups, 231--250, Springer Proc. Math., 10, Springer, New York, 2012.

\end{thebibliography}
\end{document}